\documentclass[12pt, reqno]{amsart}
\usepackage[utf8]{inputenc}
\usepackage{amsmath,color}
\usepackage{amsfonts}
\usepackage{amssymb}
\usepackage{amsthm}
\newcommand{\eps}{\varepsilon}

\newtheorem{theorem}{Theorem}[section]
\newtheorem{proposition}[theorem]{Proposition}
\newtheorem{lemma}[theorem]{Lemma}

\theoremstyle{definition}

\numberwithin{equation}{section}

\title{Analytic ranks of automorphic $L$-functions and Landau-Siegel zeros}
\author{Hung M. Bui, Kyle Pratt and Alexandru Zaharescu}
\address{Department of Mathematics, University of Manchester, Manchester M13 9PL, UK}
\email{hung.bui@manchester.ac.uk}
\address{All Souls College, Oxford OX1 4AL, UK}
\email{kyle.pratt@maths.ox.ac.uk, kvpratt@gmail.com}
\address{Department of Mathematics, University of Illinois at Urbana-Champaign, 1409 West Green Street, Urbana, IL 61801, USA
and Simion Stoilow Institute of Mathematics of the Romanian Academy, P.O. Box 1-764, RO-014700 Bucharest,
Romania}
\email{zaharesc@illinois.edu}

\allowdisplaybreaks

\subjclass[2010]{11F41, 11F66, 11M20, 11F11 \\ \indent \textit{Keywords and phrases}: Landau-Siegel zeros, exceptional characters, automorphic $L$-functions, analytic ranks, ranks of Jacobians, Birch and Swinnerton-Dyer conjecture, nonvanishing, mollfier}

\begin{document}

\maketitle

\begin{abstract}
We relate the study of Landau-Siegel zeros to the ranks of Jacobians $J_0(q)$ of modular curves for large primes $q$. By a conjecture of Brumer-Murty, the rank should be equal to half of the dimension. Equivalently, almost all newforms of weight two and level $q$ have analytic rank $\leq 1$. We show that either Landau-Siegel zeros do not exist, or that almost all such newforms have analytic rank $\leq 2$. In particular, almost all odd newforms have analytic rank equal to one. Additionally, for a sparse set of primes $q$ we show the rank of $J_0(q)$ is asymptotically equal to the rank predicted by the Brumer-Murty conjecture.
\end{abstract}

\tableofcontents

\section{Introduction}

Given a newform $f$ of weight two and prime level $q$ (i.e. $f \in S_2^*(q)$), we denote its corresponding $L$-function by $L(f,s)$. We are interested in these $L$-functions because, among other reasons, the $L$-function of the Jacobian of the curve $X_0(q)$ is given by a product over $L$-functions $L(f,s)$ (see [\textbf{\ref{Shim}}]),
\begin{equation}\label{Shimura}
L(J_0(q),s)=\prod_{f\in S_2^*(q)}L(f,s).
\end{equation}
A generalization of the Birch and Swinnerton-Dyer conjecture then relates the order of vanishing of $L(f,s)$ at $s = 1/2$ to the algebraic structure of the Jacobian of $X_0(q)$. We define the \emph{analytic rank} $r_f$ to be the order of vanishing of $L(f,s)$ at $s = 1/2$.

The $L$-function $L(f,s)$ is self-dual, and therefore its root number $\varepsilon_f$ satisfies $\varepsilon_f \in \{\pm 1\}$. If the root number is $+1$ then we say that $f$ or $L(f,s)$ is \emph{even}, and if the root number is $-1$ then we say that $f$ or $L(f,s)$ is \emph{odd}. Accordingly, we may speak of the parity of $f$ or of $L(f,s)$.

The parity of $f$ has an effect on $r_f$. Specifically, if $f$ is even then $r_f$ is even, and if $f$ is odd then $r_f$ is odd. As $r_f$ is a nonnegative integer, we observe that if $r_f$ is odd then $r_f \neq 0$. 
As $q \rightarrow \infty$ asymptotically half of the elements of $S_2^*(q)$ are even and the other half are odd. Brumer [\textbf{\ref{Br}}] and Murty [\textbf{\ref{Mur}}] conjectured that almost all even forms $f$ have $r_f=0$, and that almost all odd forms $f$ have $r_f=1$. Thus, we expect that $50\%$ of the forms have analytic rank zero, $50\%$ of the forms have analytic rank one, and $0\%$ of the forms have analytic rank two or more.

This conjecture, if true, has important consequences. For example, Iwaniec and Sarnak related the proportion of forms with analytic rank zero to the Landau-Siegel zero problem [\textbf{\ref{IwaSar}}]. In particular, they showed that if one could show strictly more than $50\%$ of the \emph{even} forms were nonzero and the central values are not too small (this is an important condition), then one could rule out the existence of Landau-Siegel zeros. We expect almost all even forms to have analytic rank zero and to have central values that are not too small [\textbf{\ref{KeatSna}}], but there seems to be a barrier to going past $50\%$.

In some sense, we approach things from another, or converse, direction. Assuming the \emph{existence} of Landau-Siegel zeros, what can one deduce about vanishing or nonvanishing of central values of automorphic $L$-functions? Roughly speaking, under such a hypothesis, we can confirm the ``odd part'' of the Brumer-Murty conjecture, and we can establish the ``even part'' of the conjecture up to a factor of two.

\begin{theorem}\label{Maintheorem}
Let $D$ be large and let $\psi$ be a real, odd, primitive Dirichlet character modulo $D$. Then one of the following two possibilities must hold:

\noindent \emph{(A)} $L(1,\psi) \geq (\log D)^{-50}$, or 

\noindent \emph{(B)} for every fixed $C \geq 750$ and every prime $q$ satisfying $D^{750} \leq q \leq D^{C}$ we have
\begin{align*}
\textup{rank}(J_0(q)) = \bigg(\frac{1}{2} + O\bigg(\sqrt{\frac{\log\log\log q}{\log\log q}} \bigg)\bigg) \textup{dim}(J_0(q))
\end{align*}
if $\psi(q) = 1$, and
\begin{align*}
\bigg(\frac{1}{2} + O\Big(\frac{\log\log\log q}{\log\log q} \Big)\bigg) \textup{dim}(J_0(q))\leq \textup{rank}(J_0(q)) \leq \bigg(1 + O\bigg(\sqrt{\frac{\log\log\log q}{\log\log q}} \bigg)\bigg) \textup{dim}(J_0(q))
\end{align*}
if $\psi(q) = -1$.
\end{theorem}

This theorem surpasses what is known even under the Generalized Riemann Hypothesis [\textbf{\ref{ILS}}]. Kowalski, Michel and VanderKam [\textbf{\ref{KMV}}] showed unconditionally that $$\text{rank}(J_0(q)) \leq (c + o(1)) \text{dim}( J_0(q)),$$ where $c=1.1891$. There appears to be a significant barrier to proving such a result with $c<1$.

Theorem \ref{Maintheorem} is a direct consequence of the following theorem.

\begin{theorem}\label{Maintheorem2}
Let $C \geq 750$ be a fixed real number. Let $D$ be large and let $\psi$ be a real, odd, primitive Dirichlet character modulo $D$. Then for any $\varepsilon > 0$ and any prime $q$ satisfying
\begin{align*}
D^{750} \leq q \leq D^{C}
\end{align*}
we have 
\begin{align*}
\frac{1}{|S_2^*(q)|}\sum_{\substack{f \in S_2^*(q) \\ r_f \leq 1}}1 = 1 + O_\varepsilon( L(1,\psi)(\log q)^{45+\varepsilon}) + O_\varepsilon( L(1,\psi)^2(\log q)^{56+\varepsilon}) + O_A((\log q)^{-A})
\end{align*}
if $\psi(q) = 1$, and
\begin{align*}
\frac{1}{|S_2^*(q)|}\sum_{\substack{f \in S_2^*(q) \\ r_f \leq 2}}1 &= 1+ O \left(\frac{\log \log (1/L(1,\psi)\log D)}{\log (1/L(1,\psi)\log D)} \right) + O_\varepsilon( L(1,\psi)(\log q)^{45+\varepsilon})\\
&\qquad\qquad + O_\varepsilon( L(1,\psi)^2(\log q)^{56+\varepsilon}) 
\end{align*}
if $\psi(q) = -1$ and $L(1,\psi)(\log D)=o(1)$.
\end{theorem}

\noindent\textbf{Remark.} For convenience we work only with weight two forms, but one could prove an analogue of Theorem \ref{Maintheorem2} in which the weight is any fixed, even $k\geq 2$.

\section{Set up and main propositions}

Let $S_2^*(q)$ be the set of primitive Hecke eigenforms of weight $2$ and level $q$ ($q$ prime). Throughout we let $D$ be large and let $\psi$ be a real, odd, primitive Dirichlet character modulo $D$. We think of $D$ as being quite small compared to $q$. We assume throughout that $L(1,\psi)(\log D) = o(1)$, where $o(1)$ denotes a quantity that tends to zero as $D \rightarrow \infty$. We let $\varepsilon$ denote a sufficiently small positive constant, the size of which might change from one line to the next.

For $f\in S_2^*(q)$, the Fourier expansion of $f$ at infinity takes the form
\[
f(z)=\sum_{n\geq1}\sqrt{n}\lambda_f(n)e(nz)
\]
with $\lambda_f(1)=1$. The $L$-function associated to $f$, $$L(f,s)=\sum_{n\geq1}\frac{\lambda_f(n)}{n^s},$$ satisfies a functional equation
\begin{align}\label{fe1}
\Lambda\Big(f,\frac12+s\Big)&:=\Big(\frac{\sqrt{q}}{2\pi}\Big)^{s}\Gamma(1+s)L\Big(f,\frac12+s\Big)\nonumber\\
& =\epsilon_f\Lambda\Big(f,\frac12-s\Big),
\end{align}
where
$$ \epsilon_f= q^{1/2}\lambda_f(q)=\pm1.$$ 

An eigenform $f$ is said to be {\it even} (resp. {\it odd}) if $\epsilon_f=1$ (resp. $\epsilon_f=-1$). As can be seen from the functional equation \eqref{fe1}, this is the same parity as the analytic rank of the $L$-function. 

For $\psi$ a primitive character modulo $D$, the twist of $f$ by $\psi$,
\[
(f\otimes\psi)(z)=\sum_{n\geq1}\sqrt{n}\psi(n)\lambda_f(n)e(nz),
\]
is a cuspidal modular form of level $qD^2$. Given that $(D,q)=1$, this is a primitive form. The twisted $L$-function is defined by $$L(f\otimes\psi,s)=\sum_{n\geq1}\frac{\psi(n)\lambda_f(n)}{n^s}.$$ This has a functional equation [\textbf{\ref{L}}]
\begin{align*}
\Lambda\Big(f\otimes\psi,\frac12+s\Big)&:=\Big(\frac{\sqrt{q}D}{2\pi}\Big)^{s}\Gamma(1+s)L\Big(f\otimes\psi,\frac12+s\Big)\\
& =-\psi(q)\epsilon_f\Lambda\Big(f\otimes\psi,\frac12-s\Big).
\end{align*}
Hence
\begin{align}\label{fe}
\Lambda_{f,\psi}\Big(\frac12+s\Big)&:=\Lambda\Big(f,\frac12+s\Big)\Lambda\Big(f\otimes\psi,\frac12+s\Big)\nonumber\\
& =-\psi(q)\Lambda_{f,\psi}\Big(\frac12-s\Big).
\end{align}
We are interested in the nonvanishing of $\Lambda_{f,\psi}'(1/2)$ and $\Lambda_{f,\psi}''(1/2)$. Studying derivatives of the product of the completed $L$-functions bestows certain technical advantages compared to studying products of derivatives of the $L$-functions themselves.

\subsection{The mollifier}

The proof of Theorem \ref{Maintheorem} goes through studying nonvanishing of central values of $L$-functions. The relevant connection is given by \eqref{Shimura}, which is due to Eichler and Shimura [\textbf{\ref{Shim}}]. We study the first and second mollified moments of $\Lambda_{f,\psi}'(1/2)$ and $\Lambda_{f,\psi}''(1/2)$, and then deduce results on nonvanishing via a simple application of Cauchy-Schwarz's inequality.

From examining Euler products, one posits that a reasonable choice of mollifier is
\[
\sum_{mn\leq X} \frac{\alpha(m)\mu(m)\alpha(n)\mu(n)\psi(n)\lambda_f(m)\lambda_f(n)}{\sqrt{mn}}
\]
with $X\asymp q$ and 
\[
\alpha(n)=\prod_{p|n}\bigg(1+\frac{1}{p}\bigg)^{-1}.
\]
Using the Hecke relation in Lemma \ref{ortho} below this is
\begin{align*}
&\sum_{mn\leq X}\frac{\alpha(m)\mu(m)\alpha(n)\mu(n)\psi(n)\lambda_f(mn)}{\sqrt{mn}}\sum_{\substack{d\leq \sqrt{X/mn}\\(d,mn)=1}}\frac{\alpha(d)^2\mu(d)^2\psi(d)}{d}.
\end{align*}
By lacunarity, we can extend the sum on $d$ to all $(d,mn)=1$, so we choose the mollifier to be
\begin{equation}\label{mollifier}
M_{f,\psi}=\sum_{a\leq X}\frac{\rho_1(a)\lambda_f(a)}{\sqrt{a}}, 
\end{equation}
where
\[
\rho_1(a)=\big((\alpha\mu)\star(\alpha\mu\psi)\big)(a)h(a)
\]
and
\[
h(n)=h_\psi(n)=\prod_{p \mid n} \left(1 + \frac{\alpha(p)^2\psi(p)}{p} \right)^{-1}.
\]
In particular we have $\rho_1(a)\ll_\eps (\log q)^\eps \tau(a)$.

Applying Lemma \ref{ortho} one more time we get
\begin{align}\label{mollifiersquare}
M_{f,\psi}^2&=\sum_{d\leq X}\frac1d\sum_{a_1,a_2\leq X/d}\frac{\rho_1(da_1)\rho_1(da_2)\lambda_f(a_1a_2)}{\sqrt{a_1a_2}}\nonumber\\
&=\sum_{a\leq X^2}\frac{\rho_2(a)\lambda_f(a)}{\sqrt{a}},
\end{align}
where
\begin{align}\label{rho2}
\rho_2(a)&=\sum_{d\leq X}\frac1d\sum_{\substack{a_1a_2=a\\a_1,a_2\leq X/d}}\rho_1(da_1)\rho_1(da_2).
\end{align}
In particular we have $\rho_2(a)\ll_\eps (\log q)^{4+\eps} \tau_4(a)$.

\subsection{Main propositions}

We state here the key propositions we need. For technical convenience we work with the so-called ``harmonic average.'' For complex numbers $\alpha_f$ we write
\begin{align*}
\sideset{}{^h}\sum_{f\in S_2^*(q)} \alpha_f = \sum_{f\in S_2^*(q)} \omega_f \alpha_f,
\end{align*}
where $\omega_f = 1/4 \pi \langle f,f \rangle$ and $\langle \cdot,\cdot \rangle$ is the Petersson inner product on $\Gamma_0(q) \backslash \mathbb{H}$. We shall discuss in Section \ref{sec:proof of main thm} how to remove the weights $\omega_f$ and thereby state results for the ``natural average''
\begin{align*}
\frac{1}{|S_2^*(q)|}\sum_{f\in S_2^*(q)} \alpha_f.
\end{align*}
We note that
\begin{align*}
\sideset{}{^h}\sum_{f\in S_2^*(q)} 1 = 1+O(q^{-3/2}).
\end{align*}

\begin{proposition}\label{mainprop1}
Provided that $q,X\geq D^8$ we have 
\begin{align*}
\sideset{}{^h}\sum_{f\in S_2^*(q)}\Lambda_{f,\psi}'\Big(\frac{1}{2}\Big)M_{f,\psi}&=2\big(1+\psi(q)\big)\mathfrak{S}_1L'(1,\psi)+O_\eps\big(L(1,\psi)(\log q)^{7+\eps}\big)+O_A\big((\log q)^{-A}\big)\\
&\qquad\qquad+O_\varepsilon\big(q^{-1/2+\varepsilon}DX\big)\nonumber
\end{align*}
and
\begin{align*}
&\sideset{}{^h}\sum_{f\in S_2^*(q)}\Lambda_{f,\psi}''\Big(\frac{1}{2}\Big)M_{f,\psi}\\
&\qquad=4\big(1-\psi(q)\big)\mathfrak{S}_1\bigg(1 + O \Big(\frac{\log \log (1/L(1,\psi)\log D)}{\log (1/L(1,\psi)\log D)} \Big) \bigg)\Big(L''(1,\psi)+(\log Q-2\gamma)L'(1,\psi)\Big)\\
&\qquad\qquad+O_\varepsilon\big(L(1,\psi)(\log q)^{8+\varepsilon}\big)+O_\varepsilon\big(q^{-1/2+\varepsilon}DX\big),
\end{align*}
where
\begin{equation}\label{singularseries1}
\mathfrak{S}_1=\prod_{\substack{p \mid D}}\bigg(1+\frac1p\bigg)^{-1} \prod_{\substack{p \leq X \\ p\nmid D}} \left(\Big(1+\frac1p\Big)^2 + \frac{\psi(p)}{p} \right)^{-1}\bigg(1- \frac{\psi(p)}{p}\bigg).
\end{equation}
\end{proposition}

\begin{proposition}\label{mainprop2}
Provided that $X\geq D^{24}$ and $D^8X^6\ll q^{1-\eps}$ we have 
\begin{align*}
\sideset{}{^h}\sum_{f\in S_2^*(q)}\Lambda_{f,\psi}'\Big(\frac{1}{2}\Big)^2M_{f,\psi}^2&=4\big(1+\psi(q)\big)^2\mathfrak{S}_2L'(1,\psi)^2+O_\eps\big(L(1,\psi)(\log q)^{37+\eps}\big)+ O_A((\log q)^{-A})\\
&\qquad\qquad+O_\varepsilon\big(q^{-1/12+\eps}D^{5/3}X^{5/2}\big)+O_\eps\big(q^{-1/4+\eps}D^{17/4}X^{19/4}\big),\nonumber
\end{align*}
where
\begin{equation}\label{singularseries2}
\mathfrak{S}_2=\mathfrak{S}_1^2=\prod_{\substack{p \mid D}}\bigg(1+\frac1p\bigg)^{-2} \prod_{\substack{p \leq X \\ p\nmid D}} \left(\Big(1+\frac1p\Big)^2 + \frac{\psi(p)}{p} \right)^{-2}\bigg(1- \frac{\psi(p)}{p}\bigg)^2.
\end{equation}
\end{proposition}

\begin{proposition}\label{mainprop3}
Provided that $X\geq D^{24}$ and $D^8X^6\ll q^{1-\eps}$ we have 
\begin{align*}
&\sideset{}{^h}\sum_{f\in S_2^*(q)}\Lambda_{f,\psi}''\Big(\frac{1}{2}\Big)^2M_{f,\psi}^2\\
&\ =16\big(1-\psi(q)\big)^2 \mathfrak{S}_2\bigg(1 + O \Big(\frac{\log \log (1/L(1,\psi)\log D)}{\log (1/L(1,\psi)\log D)} \Big) \bigg)\Big(L''(1,\psi)+(\log Q-2\gamma)L'(1,\psi)\Big)^2\\
&\qquad\qquad+ O_\varepsilon( L(1,\psi)(\log q)^{39+\varepsilon}) +O_\varepsilon\big(q^{-1/12+\eps}D^{5/3}X^{5/2}\big)+O_\eps\big(q^{-1/4+\eps}D^{17/4}X^{19/4}\big).
\end{align*}
\end{proposition}

\section{Auxiliary lemmas}

The first lemma is the Hecke relation.

\begin{lemma}\label{ortho}
For $m,n\geq1$, we have
\begin{displaymath}
\lambda_{f}(m)\lambda_{f}(n)=\sum_{\substack{d|(m,n)\\(d,q)=1}}\lambda_{f}\Big(\frac{mn}{d^2}\Big).
\end{displaymath}
\end{lemma}

Let $S(m,n;c)$ be the Kloosterman sum,
\begin{equation*}
S(m,n;c)=\sideset{}{^*}\sum_{u(\text{mod}\
c)}e\Big(\frac{mu+n\overline{u}}{c}\Big),
\end{equation*}
 and given a character $\chi$ we denote by $S_\chi(m,n;c)$ the hybrid Gauss-Kloosterman sum,
\[
S_\chi(m,n;c)=\sideset{}{^*}\sum_{u(\text{mod}\
c)}\chi(u)e\Big(\frac{mu+n\overline{u}}{c}\Big).
\]

\begin{lemma}\label{Petersson}
For $m,n\geq1$, we have
\begin{equation*}
\sideset{}{^h}\sum_{f\in S_2^*(q)}\lambda_{f}(m)\lambda_{f}(n)=\mathbf{1}_{m=n}-\frac{2\pi}{ q}\sum_{c\geq1}\frac{S(m,n;cq)}{c}J_{1}\Big(\frac{4\pi\sqrt{mn}}{cq}\Big).
\end{equation*}
As a result,
\[
\sideset{}{^h}\sum_{f\in S_2^*(q)}\lambda_{f}(m)\lambda_{f}(n)=\mathbf{1}_{m=n}+O_\varepsilon\big( q^{-3/2}(m,n,q)^{1/2}(mn)^{1/2+\varepsilon}\big).
\]
\end{lemma}
\begin{proof}
The first statement is a particular case of Petersson’s trace formula. The last estimate follows easily from the bound $J_1(x)\ll x$ and the Weil bound on Kloosterman sums.
\end{proof}

\begin{lemma}\label{Nelsonlemma}
Let $\chi$ be a character modulo $d$. Let $c=c_1c_2$ and $m=m_1m_2$, where $c_1m_1|d^\infty$ and $(c_2m_2,d)=1$. Then $S_{\chi}(m,0;c)$ vanishes unless $dm_1 = c_1$, and in that case we have
\begin{equation*}
S_{\chi}(m,0;c)=m_1\tau(\chi)\chi(c_2m_2)\sum_{\substack{r|(c_2,m_2)}}\mu\Big(\frac{c_2}{r}\Big)r.
\end{equation*}
\end{lemma}
\begin{proof}
See [\textbf{\ref{N}}; Lemma 2.7].
\end{proof}

The next lemma is the approximate functional equation for $\Lambda_{f,\psi}^{(k)}(1/2)$.

\begin{lemma}\label{afe}
Let $k\in\mathbb{N}$. Let $G(u)$ be an even entire function of rapid decay in any fixed strip $|\emph{Re}(u)| \leq C$ satisfying $G(0)=1$, $G(1)=0$ and $G^{(j)}(0)=0$ for $1\leq j\leq k$. Let
\begin{equation}\label{formulaV}
V_k(x)=\frac{1}{2\pi i}\int_{(1)}G(u)\Gamma(1+u)^2x^{-u}\frac{du}{u^{k+1}}.
\end{equation}
Then we have
\begin{align*}
\Lambda_{f,\psi}^{(k)}\Big(\frac{1}{2}\Big)=k!\big(1+(-1)^{k+1}\psi(q)\big)\sum_{(d,q)=1}\frac{\psi(d)}{d}\sum_{n}\frac{(1\star\psi)(n)\lambda_f(n)}{\sqrt{n}}V_{k}\Big(\frac{d^2n}{Q}\Big),
\end{align*}
where
\[
Q=\frac{qD}{4\pi^2}.
\]
\end{lemma}
\begin{proof}
As $G(0)=1$ and $G^{(j)}(0)=0$ for $1\leq j\leq k$, by Cauchy's theorem we have
\[
\frac{1}{2\pi i}\int_{(1)}G(u)\Lambda_{f,\psi}\Big(\frac{1}{2}+u\Big)\frac{du}{u^{k+1}}=\frac{\Lambda_{f,\psi}^{(k)}(1/2)}{k!}+\frac{1}{2\pi i}\int_{(-1)}G(u)\Lambda_{f,\psi}\Big(\frac{1}{2}+u\Big)\frac{du}{u^{k+1}}.
\]
By a change of variables $u\rightarrow-u$ and using \eqref{fe}, we then obtain
\begin{displaymath}
\frac{\Lambda_{f,\psi}^{(k)}(1/2)}{k!}=\frac{1+(-1)^{k+1}\psi(q)}{2\pi i}\int_{(1)}G(u)\Lambda_{f,\psi}\Big(\frac{1}{2}+u\Big)\frac{du}{u^{k+1}}.
\end{displaymath}
Writing $\Lambda_{f,\psi}$ in terms of Dirichlet series and then integrating term-by-term gives
\begin{align*}
\Lambda_{f,\psi}^{(k)}\Big(\frac{1}{2}\Big)=k!\big(1+(-1)^{k+1}\psi(q)\big)\sum_{m,n}\frac{\psi(n)\lambda_f(m)\lambda_f(n)}{\sqrt{mn}}V_{k}\Big(\frac{mn}{Q}\Big).
\end{align*}
The lemma now follows by applying Lemma \ref{ortho}.
\end{proof}

\noindent\textbf{Remarks.}
\begin{itemize}
\item We can move the line of integration in \eqref{formulaV} to $\text{Re}(u)=u_0$ for any fixed $u_0>0$ without changing the value of $V_k(x)$.
\item For $k\leq 2$ an admissible choice of $G$ in Lemma \ref{afe} is
\[
G(u)=e^{u^6}(u^4-1)^2,
\] say, but there is no need to specify $G$. We have not used the condition $G(1)=0$, but this is to cancel out certain poles of some Gamma functions, which appear in the evaluation of the off-off-diagonal terms in Section \ref{afeood}. This substantially simplifies our later calculations.
\end{itemize}

The next result is fundamental to our work. It is a quantitative statement of ``lacunarity''.

\begin{lemma}\label{lem:basic lac estimate}
For any $x \geq D^2$ we have
\begin{align*}
\sum_{D^2 < n \leq x}\frac{(1\star\psi)(n)}{n}\leq L(1,\psi)(\log x).
\end{align*}
\end{lemma}
\begin{proof}
Apply [\textbf{\ref{IK}}; (22.109)] twice and subtract the expressions. This gives
\begin{align*}
\sum_{D^2 < n \leq x}\frac{(1\star\psi)(n)}{n}\leq L(1,\psi)(\log x) - L(1,\psi)(\log D^2) + O \left(D^{-3/4}(\log D) \right) \leq L(1,\psi)(\log x),
\end{align*}
since $L(1,\psi) \gg D^{-1/2}$.
\end{proof}

We often encounter sums where $(1\star \psi)$ is twisted by divisor functions. The next lemma provides a bound for such sums.

\begin{lemma}\label{lem:updated lacunarity lemma}
%
Let $k\in\mathbb{N}$, $B>0$ and $N\geq 8$ be fixed. Let $D^{8} \leq Z \leq D^N$ and define $T = \exp((\log Z)(\log \log Z)^2)$. Then we have
\begin{align*}
\sum_{\substack{D^{8} < n \leq T^B}} \frac{(1\star\psi)(n)\tau(n)^k}{n} \ll_{\varepsilon,A,B,k,N} L(1,\psi)(\log D)^{2^{k+1}+1+\varepsilon} + (\log D)^{-A}
\end{align*}
for any fixed $A>0$.
\end{lemma}
\begin{proof}
Let $\mathcal{T}$ denote the sum we wish to bound. Let $Z_0 = \exp((\log Z)(\log \log Z)^{-2})$ and factor $n = n_0 n_1$, where $P^+(n_0) \leq Z_0$ and $P^-(n_1) > Z_0$. Then
\begin{align*}
\mathcal{T}&\ll \sum_{\substack{D^{8} < n_0n_1 \leq T^B \\ p \mid n_0 \Rightarrow p \leq Z_0 \\ p \mid n_1 \Rightarrow p > Z_0}}\frac{(1\star\psi)(n_0)\tau(n_0)^k(1\star\psi)(n_1)\tau(n_1)^k}{n_0n_1}.
\end{align*}
Since $n > D^{8}$ we must have that one of $n_0$ or $n_1$ exceeds $D^{4}$. To bound the contribution from $n_0 > D^{4}$ we set $\delta =1/\log Z_0$ and use Rankin's trick to obtain
\begin{align*}
&\ll (\log \log Z)^{O_k(1)}\mathop{\sum}_{\substack{D^{4} < n_0 \leq T^B \\ p \mid n_0 \Rightarrow p \leq Z_0}}\frac{(1\star\psi)(n_0)\tau(n_0)^k}{n_0}\ll \frac{(\log \log Z)^{O_k(1)}}{D^{4\delta}}\prod_{p \leq Z_0}\left(1 + \frac{2^{k+1} p^\delta}{p} \right)\\
&\ll\frac{(\log \log Z)^{O_k(1)}}{D^{4\delta}}\prod_{p \leq Z_0}\left(1 + \frac{2^{k+1} e}{p} \right)\ll \frac{(\log Z_0)^{O_k(1)}}{D^{4\delta}}\ll_{A,k} (\log Z)^{-A}.
\end{align*}
This contribution is clearly acceptable, and henceforth we may assume that in $\mathcal{T}$ we have $n_1 > D^{4}$. Actually, it is technically convenient to factor $n_1 = ms$, say, where $m$ is squarefree, $s$ is squarefull and $(m,s)=1$. The contribution from $s > D^{2}$ is easily seen to be $O(Z_0^{-1})$, so we may assume $m > D^2$, and therefore
\begin{align}\label{eq:intermediate bound on mathfrak T in updated lacunarity lemma}
\mathcal{T}&\ll_{A,k} (\log Z_0)^{2^{k+1}}\mathop{\sum}_{\substack{D^2 < m \leq T^B \\ p \mid m \Rightarrow p > Z_0}}\frac{\mu^2(m)(1\star\psi)(m)\tau(m)^k}{m}+(\log Z)^{-A} .
\end{align}

We introduce a quantity
\begin{align*}
\varpi &:= \frac{\log\log Z}{(\log\log\log Z)^{1/2}}
\end{align*}
and bound the sum on $m$ in \eqref{eq:intermediate bound on mathfrak T in updated lacunarity lemma} according to whether $\omega(m) \leq \varpi$ or $\omega(m) > \varpi$. The motivation for this division is as follows. If $\omega(m)\leq \varpi$, then the divisor function $\tau(m)$ can be controlled easily because $m$ does not have so many prime factors. On the other hand, if $\omega(m) > \varpi$ then $m$ has many more prime factors than usual and this is therefore a rare event: the typical integer $m$ with $m \leq T^B$ and $P^-(m) > Z_0$ has $\approx \log\log \log Z$ prime factors.

We consider first the case where $\omega(m)\leq \varpi$. Since $m$ is squarefree we have $\tau(m) = 2^{\omega(m)}\leq 2^\varpi$, and therefore the contribution from those $m$ with $\omega(m)\leq \varpi$ is
\begin{align*}
\ll 2^{k\varpi}(\log Z)^{2^{k+1}}\sum_{D^2 < m \leq T^B}\frac{(1\star\psi)(m)}{m}\ll_{\varepsilon,B,k}L(1,\psi) (\log Z)^{2^{k+1}+1+\varepsilon},
\end{align*}
where the last inequality follows from Lemma \ref{lem:basic lac estimate}.

The contribution to $\mathcal{T}$ from those $m$ with $\omega(m)> \varpi$ is
\begin{align*}
&\ll (\log Z)^{O_k(1)}\sum_{\substack{m \leq T^B \\ p\mid m \Rightarrow p > Z_0 \\ \omega(m) > \varpi}}\frac{\mu^2(m)2^{\omega(m)(k+1)}}{m}\leq (\log Z)^{O_k(1)}\sum_{t > \varpi}\frac{1}{t!}\Bigg(\sum_{Z_0 < p \leq T^B}\frac{2^{k+1}}{p} \Bigg)^t\\
&=(\log Z)^{O_k(1)}\sum_{t > \varpi}\frac{1}{t!}\left(2^{k+3}\log\log\log Z + 2^{k+1}\log B + o_k(1) \right)^t\\
&\ll_{B,k}(\log Z)^{O_k(1)}\sum_{t > \varpi}\frac{(2^{k+4}\log\log\log Z)^t}{t!}\ll_{A,B,k}(\log Z)^{-A}.
\end{align*}
We therefore have
\begin{align*}
\mathcal{T}\ll_{\varepsilon,A,B,k} L(1,\psi)(\log Z)^{2^{k+1}+1+\varepsilon} + (\log Z)^{-A},
\end{align*}
which completes the proof.
\end{proof}

In order to bound some logarithmic derivatives we will need estimates for certain sums over primes. The key point is that lacunarity enables a small amount of savings over the trivial bound.

\begin{lemma}\label{lem:lac sum over primes}
For any $x \geq D$ we have
\begin{align*}
\sum_{\substack{p \leq x \\ \psi(p)=1}}\frac{\log p}{p}\ll \frac{\log\log(1/L(1,\psi)\log D)}{\log(1/L(1,\psi)\log D)}\log D + L(1,\psi)^{1/2} (\log x)^{3/2}.
\end{align*}
\end{lemma}
\begin{proof}
We use lacunarity to reduce the range of $p$ to $p\leq D$. Since our basic lacunarity estimate Lemma \ref{lem:basic lac estimate} does not quite apply in this shorter range, we use a crude ``tensor power'' trick. We have
\begin{align*}
\Big(\sum_{\substack{D < p\leq x \\ \psi(p)=1}} \frac{1}{p}\Big)^2&\leq \sum_{D < p_1,p_2\leq x}\frac{\big(1+\psi(p_1)\big)\big(1+\psi(p_2)\big)}{p_1p_2}\leq \sum_{p > \sqrt{D}}\frac{4}{p^2} + \sum_{D^2 < n\leq x^2}\frac{(1\star\psi)(n)}{n}\\
&\ll D^{-1/2}(\log D)^{-1} + L(1,\psi)(\log x)\ll L(1,\psi)(\log x),
\end{align*}
where in the last inequality we have used the lower bound $L(1,\psi)\gg D^{-1/2}$. Therefore
\begin{align*}
\sum_{\substack{p \leq x \\ \psi(p)=1}}\frac{\log p}{p}=\sum_{\substack{p \leq D \\ \psi(p)=1}}\frac{\log p}{p} + O\big(L(1,\psi)^{1/2}(\log x)^{3/2}\big).
\end{align*}
We now define
\begin{align*}
\kappa =  \frac{4\log\log(1/L(1,\psi)\log D)}{\log(1/L(1,\psi)\log D)}
\end{align*}
and observe that we may assume $\kappa < 1/20$, say. We split the sum over $p\leq D$ as
\begin{align*}
\sum_{\substack{p \leq D \\ \psi(p)=1}}\frac{\log p}{p} &= \sum_{\substack{p \leq D^\kappa \\ \psi(p)=1}}\frac{\log p}{p} + \sum_{\substack{D^\kappa < p \leq D \\ \psi(p)=1}}\frac{\log p}{p}.
\end{align*}
For the first sum we use the trivial bound
\begin{align*}
\sum_{\substack{p \leq D^\kappa \\ \psi(p)=1}}\frac{\log p}{p} \ll \kappa \log D.
\end{align*}
For the other sum we use Proposition 3.1 of [\textbf{\ref{FI13}}], the proof of which utilizes a more sophisticated version of the tensor power trick above. We obtain the bound
\begin{align*}
\sum_{\substack{D^\kappa < p \leq D \\ \psi(p)=1}}\frac{\log p}{p}\ll (\log D)\sum_{D^\kappa < p \leq D} \frac{(1\star\psi)(p)}{p} &\ll \kappa^{-1} \big(L(1,\psi) \log D\big)^{\kappa/2} (\log D).
\end{align*}
It follows that
\begin{align*}
\sum_{\substack{p \leq x \\ \psi(p)=1}}\frac{\log p}{p}&\ll \frac{\log\log(1/L(1,\psi)\log D)}{\log(1/L(1,\psi)\log D)}\log D + L(1,\psi)^{1/2} (\log x)^{3/2},
\end{align*}
as desired.
\end{proof}

\begin{lemma}\label{lem:deriv expressions}
We have
\begin{align*}
L'(1,\psi) &= \sum_{n \leq D^2} \frac{(1\star\psi)(n)}{n} + O\big(L(1,\psi)(\log D)\big)
\end{align*}
and
\begin{align*}
L''(1,\psi) &= -2\gamma L'(1,\psi) - 2\sum_{n \leq D^2} \frac{(1\star\psi)(n)(\log n)}{n} + O\big(L(1,\psi)(\log D)^2\big).
\end{align*}
\end{lemma}
\begin{proof}
The expression for $L'(1,\psi)$ comes from using [\textbf{\ref{IK}}; (22.109)] and the lower bound $L(1,\psi) \gg D^{-1/2}$. The expression for $L''(1,\psi)$ comes from [\textbf{\ref{IK}}; Exercise 2 on p.527] and a similar argument.
\end{proof}

\section{The first moments - Initial manipulations}\label{1stinitial}

In this section we make some initial manipulations in the first moment calculation. The error term analysis is easy, and the bulk of the effort goes into analysing the main term. This is in direct contrast to the second moment, which we discuss later.

From Lemma \ref{afe} we have
\begin{align*}
\sideset{}{^h}\sum_{f\in S_2^*(q)}\Lambda_{f,\psi}^{(k)}\Big(\frac{1}{2}\Big)M_{f,\psi}&=k!\big(1+(-1)^{k+1}\psi(q)\big)\sum_{(d,q)=1}\frac{\psi(d)}{d}\\
&\qquad\qquad\sum_{a\leq X}\sum_{n}\frac{\rho_1(a)(1\star\psi)(n)}{\sqrt{an}}V_k\Big(\frac{d^2n}{Q}\Big)\sideset{}{^h}\sum_{f\in S_2^*(q)}\lambda_f(a)\lambda_f(n).
\end{align*}
The condition $(d,q)=1$ may be removed at a negligible cost due to the decay of the function $V_k$. In view of Lemma \ref{Petersson} and \eqref{formulaV}, this is equal to
\begin{align*}
&k!\big(1+(-1)^{k+1}\psi(q)\big)\sum_{d}\frac{\psi(d)}{d}\sum_{n\leq X}\frac{\rho_1(n)(1\star\psi)(n)}{n}V_k\Big(\frac{d^2n}{Q}\Big)+O_\varepsilon\big(q^{-1/2+\varepsilon}DX\big)\nonumber\\
&\ =k!\big(1+(-1)^{k+1}\psi(q)\big)\frac{1}{2\pi i}\int_{(1)}G(u)\Gamma(1+u)^2Q^uL(1+2u,\psi)\mathcal{S}_1(u)\frac{du}{u^{k+1}}+O_\varepsilon\big(q^{-1/2+\varepsilon}DX\big),
\end{align*}
where
\begin{align}\label{eq:defn of S1 u}
\mathcal{S}_1(u)&=\sum_{n\leq X}\frac{\rho_1(n)(1\star\psi)(n)}{n^{1+u}}.
\end{align}

Note that $\mathcal{S}_1(u)\ll_\eps X^{-\text{Re}(u)+\varepsilon}$ for $\text{Re}(u)\ll 1/\log q$ and trivially we have $\mathcal{S}_1^{(j)}(0)\ll_{\eps,j}(\log q)^{j+4+\eps}$ for any $j\geq0$. We move the line of integration to $\text{Re}(u)=-1/2+\varepsilon$, crossing a pole of order $(k+1)$ at $u=0$, and the new integral is 
\[
\ll_\eps q^{\eps}\Big(\frac{Q}{X}\Big)^{-1/2}D^{1/2}\ll_\eps q^{-1/2+\eps}X^{1/2}.
\]
 The contribution of the residue is
\[
2\big(1+\psi(q)\big)\mathcal{S}_1(0)L'(1,\psi)+O_\eps\big(L(1,\psi)(\log q)^{5+\eps}\big)
\]
if $k=1$, and is
\begin{align*}
&4\big(1-\psi(q)\big)\mathcal{S}_1(0)\Big(L''(1,\psi)+(\log Q-2\gamma)L'(1,\psi)\Big)\\
&\qquad\qquad+4\big(1-\psi(q)\big)\mathcal{S}_1'(0)L'(1,\psi)+O_\eps\big(L(1,\psi)(\log q)^{6+\eps}\big)
\end{align*}
if $k=2$, since $\Gamma'(1)=-\gamma$. Thus we are left to study the singular series $\mathcal{S}_1(0)$ and its first derivative.

\section{The first moments - The singular series}\label{appendA}

In this section we finish the main term analysis of the first moment. It is helpful to recall that
\[
\rho_1(a)=\big((\alpha\mu)\star(\alpha\mu\psi)\big)(a)h(a),
\]
where
\[
h(n)=h_\psi(n)=\prod_{p \mid n} \left(1 + \frac{\alpha(p)^2\psi(p)}{p} \right)^{-1}
\]
and $\rho_1(a)\ll_\eps(\log q)^\eps\tau(a)$.
The following two lemmas and our results in Section \ref{1stinitial} will prove Proposition \ref{mainprop1}.

\begin{lemma}\label{lemmaA1}
Provided that $X \geq D^8$ we have
\begin{align*}
\mathcal{S}_1(0)&=\mathfrak{S}_1+O_{\eps}\big(L(1,\psi)(\log q)^{5+\eps}\big)+O_{A}\big((\log q)^{-A}\big),
\end{align*}
where $\mathfrak{S}_1$ is given by \eqref{singularseries1}.
\end{lemma}
\begin{proof}

Our basic idea is to use lacunarity to show that the sum
\begin{align*}
\mathcal{S}_1(0) &= \sum_{\substack{n\leq X}}\frac{\rho_1(n)(1\star\psi)(n)}{n}
\end{align*}
is, up to an acceptable error, equal to the Euler product
\begin{align*}
\mathcal{P}_1 &:= \prod_{p\leq X}\bigg( \sum_{n \geq 0}\frac{\rho_1(p^n)(1\star\psi)(p^n)}{p^n}\bigg).
\end{align*}
We actually start with $\mathcal{P}_1$ and show how to trim it down to get $\mathcal{S}_1(0)$.

Observe that
\begin{align*}
\mathcal{P}_1 &= \sum_{p \mid n\Rightarrow p \leq X}\frac{\rho_1(n)(1\star\psi)(n)}{n}.
\end{align*}
We first use a crude estimate to truncate the sum. Define $T = \exp((\log X)(\log \log X)^2)$ and $\delta = 1/\log X$. The contribution to $\mathcal{P}_1$ from $n > T$ is
\begin{align*}
&\leq \sum_{\substack{p \mid n\Rightarrow p \leq X \\ n > T}}\frac{|\rho_1(n)|\tau(n)}{n} \leq T^{-\delta}\sum_{p \mid n\Rightarrow p \leq X}\frac{|\rho_1(n)|\tau(n)}{n^{1-\delta}} = T^{-\delta} \prod_{p\leq X}\left(1 + \frac{2|\rho_1(p)|p^{\delta}}{p} + \frac{3|\rho_1(p^2)|p^{2\delta}}{p^2}\right)\\
&\ll T^{-\delta}\prod_{p\leq X}\left( 1+\frac{4\alpha(p)h(p)p^\delta}{p}\right)\ll T^{-\delta}\prod_{p\leq X}\left( 1+\frac{4p^\delta}{p}\right)\leq T^{-\delta}\prod_{p\leq X}\left( 1+\frac{12}{p}\right)\ll (\log X)^{12} T^{-\delta},
\end{align*}
where the penultimate inequality follows since $p^\delta \leq e < 3$. Since
\begin{align*}
T^\delta = \exp((\log \log X)^2)\gg_A (\log X)^A,
\end{align*}
we see that the contribution from $n>T$ contributes an error of size $O_{A}((\log q)^{-A})$. We have therefore deduced that
\begin{align*}
\mathcal{P}_1 &= \sum_{\substack{n \leq T \\ p \mid n\Rightarrow p \leq X}}\frac{\rho_1(n)(1\star\psi)(n)}{n}+O_{A}((\log q)^{-A}).
\end{align*}

We now use the lacunarity of $(1\star\psi)$ to replace the condition $n\leq T$ by $n\leq X$. For $n\leq T$ we have $h(n)\ll(\log q)^\varepsilon$, and therefore the contribution of $n>X$ is
\begin{align*}
&\ll_\varepsilon (\log q)^\varepsilon \sum_{\substack{X < n \leq T \\ p\mid n \Rightarrow p\leq X}}\frac{\tau(n)(1\star\psi)(n)}{n} \ll_{\varepsilon,A}L(1,\psi)(\log q)^{5+\varepsilon} + (\log q)^{-A},
\end{align*}
the second inequality following from Lemma \ref{lem:updated lacunarity lemma}. Thus
\begin{align*}
\mathcal{P}_1 &= \sum_{n\leq X}\frac{\rho_1(n)(1\star\psi)(n)}{n} +O_\varepsilon(L(1,\psi)(\log q)^{5+\varepsilon})+O_A((\log q)^{-A})\\
&=\mathcal{S}_1(0) + O_\varepsilon(L(1,\psi)(\log q)^{5+\varepsilon})+O_A((\log q)^{-A}).
\end{align*}

To finish our analysis we must evaluate
\begin{align*}
\mathcal{P}_1 &= \prod_{p\leq X} \bigg(\sum_{n \geq 0}\frac{\rho_1(p^n)(1\star\psi)(p^n)}{p^n}\bigg).
\end{align*}
There are three cases to consider, depending on the value of $\psi(p) \in \{-1,0,1\}$.

If $\psi(p) = 0$ (that is, if $p \mid D$), then
\begin{align*}
\sum_{n \geq 0}\frac{\rho_1(p^n)(1\star\psi)(p^n)}{p^n} = 1 + \frac{\alpha(p)\mu(p)}{p}=\left(1+\frac{1}{p}\right)^{-1}.
\end{align*}
If $\psi(p)=-1$, then
\begin{align*}
\sum_{n \geq 0}\frac{\rho_1(p^n)(1\star\psi)(p^n)}{p^n} &=1+\frac{\rho_1(p^2)}{p^2} = 1-\frac{\alpha(p)^2h(p)}{p^2}=\frac{p^2+p}{p^2+p+1}\\
&=\left(\Big(1+\frac1p\Big)^2 + \frac{\psi(p)}{p} \right)^{-1}\bigg(1- \frac{\psi(p)}{p}\bigg).
\end{align*}
Lastly, if $\psi(p)=1$, then
\begin{align*}
\sum_{n \geq 0}\frac{\rho_1(p^n)(1\star\psi)(p^n)}{p^n} &= 1-4\frac{\alpha(p)h(p)}{p}+3\frac{\alpha(p)^2h(p)}{p^2} = \frac{p^2-p}{p^3+3p+1}\\
&=\left(\Big(1+\frac1p\Big)^2 + \frac{\psi(p)}{p} \right)^{-1}\bigg(1- \frac{\psi(p)}{p}\bigg).
\end{align*}
\end{proof}

\begin{lemma}\label{lemmaA2}
With $\mathcal{S}_1$ given by \eqref{eq:defn of S1 u} and $X\geq D^8$ we have
\begin{align*}
\mathcal{S}_1'(0)&\ll_{\eps,A} \mathfrak{S}_1\frac{\log\log(1/L(1,\psi)\log D)}{\log(1/L(1,\psi)\log D)}\log D+L(1,\psi)(\log q)^{6+\eps}.
\end{align*}
\end{lemma}
\begin{proof}

By Cauchy's integral formula, we have
\begin{align*}
\mathcal{S}_1'(0) &= \frac{1}{2\pi i}\oint \mathcal{S}_1(u) \frac{du}{u^2},
\end{align*}
where $u$ runs over the circle centered at the origin of radius $(\log q)^{-1-\varepsilon_0}$ and $\varepsilon_0>0$ is a sufficiently small, fixed constant. By arguing precisely as in the proof of Lemma \ref{lemmaA1}, we have
\begin{align*}
\mathcal{S}_1(u) &=  \prod_{p \leq X}\bigg(\sum_{n \geq 0}\frac{\rho_1(p^n)(1\star\psi)(p^n)}{p^{n(1+u)}}\bigg)+O_\varepsilon(L(1,\psi)(\log q)^{5+\varepsilon}) + O_A((\log q)^{-A}) \\
&=\mathcal{P}_1(u)+ O_\varepsilon(L(1,\psi)(\log q)^{5+\varepsilon}) + O_A((\log q)^{-A}),
\end{align*}
say. The choice of $|u|$ ensures that $T^{|u|} = 1+o(1)$, where $T$ is as in Lemma \ref{lemmaA1}. Hence
\begin{align*}
\mathcal{S}_1'(0) &= \frac{1}{2\pi i}\oint \mathcal{P}_1(u) \frac{du}{u^2}+O_\varepsilon(L(1,\psi)(\log q)^{6+\varepsilon_0 +\varepsilon}) + O_A((\log q)^{-A}).
\end{align*}
Noting that $\varepsilon_0$ can be arbitrarily small, this gives some of the error terms in the statement of the lemma.

We have
\begin{align*}
\frac{1}{2\pi i}\oint \mathcal{P}_1(u) \frac{du}{u^2} = \mathcal{P}_1'(0) = \mathcal{Q}(0) \mathcal{P}_1(0),
\end{align*}
say, where
\begin{align*}
\mathcal{Q}(u) = \frac{\mathcal{P}_1'(u)}{\mathcal{P}_1(u)} = \frac{d}{du}\log \mathcal{P}_1(u)
\end{align*}
is a sum over primes $p\leq X$. By simple computations analogous to those in the proof of Lemma \ref{lemmaA1} we deduce that
\begin{align*}
\mathcal{P}_1(u) &= \prod_{\substack{p \mid D}}\left(1 - \frac{\alpha(p)}{p^{1+u}}\right)\prod_{\substack{p\leq X \\ \psi(p)=-1}}\left(1 - \frac{\alpha(p)^2h(p)}{p^{2(1+u)}}\right)\prod_{\substack{p\leq X \\ \psi(p)=1}}\left(1-\frac{4\alpha(p)h(p)}{p^{1+u}}+\frac{3\alpha(p)^2h(p)}{p^{2(1+u)}}\right),
\end{align*}
and by logarithmic differentiation we derive
\begin{align*}
\mathcal{Q}(0) &= \sum_{p \mid D} \frac{\alpha(p)(\log p)}{p-\alpha(p)} + \sum_{\substack{p \leq X \\ \psi(p)=-1}} \frac{2\alpha(p)^2h(p)(\log p)}{p^{2}-\alpha(p)^2h(p)}\\
&\qquad\qquad+\sum_{\substack{p \leq X \\ \psi(p)=1}}\frac{4\alpha(p)h(p)p(\log p)-6\alpha(p)^2h(p)(\log p)}{p^2-4\alpha(p)h(p)p+3\alpha(p)^2h(p)}.
\end{align*}
We observe that the denominators in all the fractions here are uniformly bounded away from zero. Since
\begin{align*}
\sum_{p\mid D}\frac{\log p}{p}\ll \log \log D
\end{align*}
and
\begin{align*}
\sum_{p}\frac{\log p}{p^2}\ll 1
\end{align*}
we see that 
\begin{align*}
\mathcal{Q}(0) &\ll \log \log D + \sum_{\substack{p \leq X \\ \psi(p)=1}}\frac{\log p}{p}.
\end{align*}
An appeal to Lemma \ref{lem:lac sum over primes} finishes the proof, once we notice that
\begin{align*}
\frac{\log\log(1/L(1,\psi)\log D)}{\log(1/L(1,\psi)\log D)} \log D \gg \log \log D
\end{align*}
by the lower bound $L(1,\psi) \gg D^{-1/2}$, and $\mathfrak{S}_1\leq 1$.
\end{proof}

\begin{proof}[Proof of Proposition \ref{mainprop1}]
By the work of Section \ref{1stinitial}, we have
\begin{align*}
\sideset{}{^h}\sum_{f\in S_2^*(q)}\Lambda_{f,\psi}'\Big(\frac{1}{2}\Big)M_{f,\psi} &= 2\big(1+\psi(q)\big)\mathcal{S}_1(0)L'(1,\psi)+O_\eps\big(L(1,\psi)(\log q)^{5+\eps}\big) + O_\varepsilon\big(q^{-1/2+\varepsilon}DX\big).
\end{align*}
Lemma \ref{lemmaA1} and the trivial bound $L'(1,\psi) \ll (\log D)^2$ together yield
\begin{align*}
\sideset{}{^h}\sum_{f\in S_2^*(q)}\Lambda_{f,\psi}'\Big(\frac{1}{2}\Big)M_{f,\psi}&=2\big(1+\psi(q)\big)\mathfrak{S}_1L'(1,\psi)+O_\eps\big(L(1,\psi)(\log q)^{7+\eps}\big)+O_A\big((\log q)^{-A}\big)\\
&\qquad\qquad+O_\varepsilon\big(q^{-1/2+\varepsilon}DX\big),
\end{align*}
which is the first part of the proposition.

We also have
\begin{align*}
&\sideset{}{^h}\sum_{f\in S_2^*(q)}\Lambda_{f,\psi}''\Big(\frac{1}{2}\Big)M_{f,\psi} = 4\big(1-\psi(q)\big)\mathcal{S}_1(0)\Big(L''(1,\psi)+(\log Q-2\gamma)L'(1,\psi)\Big)\\
&\qquad\qquad+4\big(1-\psi(q)\big)\mathcal{S}_1'(0)L'(1,\psi) +O_\eps\big(L(1,\psi)(\log q)^{6+\eps}\big)+O_\varepsilon\big(q^{-1/2+\varepsilon}DX\big).
\end{align*}
We utilize the trivial bounds $L'(1,\psi) \ll (\log D)^2, L''(1,\psi) \ll (\log D)^3$ and Lemmas \ref{lemmaA1} and \ref{lemmaA2} to derive
\begin{align}\label{simplify1}
&\sideset{}{^h}\sum_{f\in S_2^*(q)}\Lambda_{f,\psi}''\Big(\frac{1}{2}\Big)M_{f,\psi} = 4\big(1-\psi(q)\big)\mathfrak{S}_1\Big(L''(1,\psi)+(\log Q-2\gamma)L'(1,\psi)\Big) \nonumber\\
&\qquad\qquad+ O \left(\mathfrak{S}_1|L'(1,\psi)|\frac{\log\log(1/L(1,\psi)\log D)}{\log(1/L(1,\psi)\log D)}\log D \right)+ O_\varepsilon (L(1,\psi)(\log q)^{8+\varepsilon}) \\
&\qquad\qquad  + O_\varepsilon\big(q^{-1/2+\varepsilon}DX\big).\nonumber
\end{align}

We can simplify things nicely if we know something about the size of
\begin{align*}
L''(1,\psi)+(\log Q-2\gamma)L'(1,\psi).
\end{align*}
We might expect that
\begin{align*}
L''(1,\psi)+(\log Q-2\gamma)L'(1,\psi) \gg (\log Q)L'(1,\psi),
\end{align*}
at least if $L'(1,\psi)$ is not too small. We can show that this does indeed occur using lacunarity. Applying Lemma \ref{lem:deriv expressions}  we obtain
\begin{align*}
L''(1,\psi)+(\log Q-2\gamma)L'(1,\psi)&= (\log Q-4\gamma )\sum_{n \leq D^2} \frac{(1\star\psi)(n)}{n}\\
&\qquad\qquad - 2\sum_{n \leq D^2} \frac{(1\star\psi)(n)(\log n)}{n}+O(L(1,\psi)(\log q)^2).
\end{align*}
We also note that
\begin{align*}
&(\log Q-4\gamma)\sum_{n \leq D^2} \frac{(1\star\psi)(n)}{n} -2 \sum_{n \leq D^2} \frac{(1\star\psi)(n)(\log n)}{n}\\
&\qquad\qquad\geq (\log Q - 4\log D-4\gamma)\sum_{n \leq D^2} \frac{(1\star\psi)(n)}{n}\geq \frac{\log Q}{2}\sum_{n \leq D^2} \frac{(1\star\psi)(n)}{n},
\end{align*}
the last inequality holding provided $q \geq D^{8}$, say, since $Q = qD/4\pi^2$. The sum over $n$ here is $\geq 1$ since $(1\star \psi)(n)$ is nonnegative and equal to 1 at $n=1$.

Since $\mathfrak{S}_1\leq 1$, we derive
\begin{align*}
&\sideset{}{^h}\sum_{f\in S_2^*(q)}\Lambda_{f,\psi}''\Big(\frac{1}{2}\Big)M_{f,\psi}\\
&\qquad=4\big(1-\psi(q)\big)\mathfrak{S}_1\bigg(1 + O \Big(\frac{\log \log (1/L(1,\psi)\log D)}{\log (1/L(1,\psi)\log D)} \Big) \bigg)\Big(L''(1,\psi)+(\log Q-2\gamma)L'(1,\psi)\Big)\\
&\qquad\qquad +O_\varepsilon\big(L(1,\psi)(\log q)^{8+\varepsilon}\big)+O_\varepsilon\big(q^{-1/2+\varepsilon}DX\big).
\end{align*}
Note that the second term in the second line of \eqref{simplify1} has been dropped using Cauchy-Schwarz's inequality.
\end{proof}

\section{The second moments - Initial manipulations}

The second moment analysis is much more complicated than the first moment analysis. A distinction in the main term analysis arises between $k=1$ and $k=2$, necessitating some additional arguments. Furthermore, the error term analysis is now highly non-trivial, and spills over several sections. 

As before, we can remove the condition $(d,q)=1$ in the expression for $\Lambda_{f,\psi}^{(k)}(1/2)$ in Lemma \ref{afe} at a negligible cost due to the decay of the function $V_k$. We use \eqref{mollifiersquare} and apply Lemma \ref{ortho} to get
\[
\Lambda_{f,\psi}^{(k)}\Big(\frac{1}{2}\Big)M_{f,\psi}^2=k!\big(1+(-1)^{k+1}\psi(q)\big)\sum_{d,b}\frac{\psi(d)}{db}\sum_{a\leq X^2/b}\sum_{n}\frac{\rho_2(ab)(1\star\psi)(bn)\lambda_f(an)}{\sqrt{an}}V_{k}\Big(\frac{d^2bn}{Q}\Big).
\]
Making use of the recursion formula
\begin{equation}\label{recursion}
(1\star\psi)(bn)=\sum_{g|(b,n)}\mu(g)(1\star\psi)\Big(\frac{ b}{g}\Big)(1\star\psi)\Big(\frac ng\Big)
\end{equation}
 we obtain that
\begin{align*}
&\Lambda_{f,\psi}^{(k)}\Big(\frac{1}{2}\Big)M_{f,\psi}^2=k!\big(1+(-1)^{k+1}\psi(q)\big)\\
&\qquad\qquad\sum_{d,b,g}\frac{\psi(d)\mu(g)(1\star\psi)(b)}{dbg^{3/2}}\sum_{a\leq X^2/bg}\sum_{n}\frac{\rho_2(abg)(1\star\psi)(n)\lambda_f(agn)}{\sqrt{an}}V_{k}\Big(\frac{d^2bg^2n}{Q}\Big).
\end{align*}
Lemma \ref{Petersson} then yields
\[
\sideset{}{^h}\sum_{f\in S_2^*(q)}\Lambda_{f,\psi}^{(k)}\Big(\frac{1}{2}\Big)^2M_{f,\psi}^2=\mathcal{M}_k^D+O\big(|E_k|\big),
\]
where
\begin{align*}
\mathcal{M}_k^D&=(k!)^2\big(1+(-1)^{k+1}\psi(q)\big)^2\sum_{\substack{d_1,d\\abg\leq X^2}}\frac{\psi(d_1)\psi(d)\mu(g)(1\star\psi)(b)}{d_1dbg^{3/2}}\\
&\qquad\qquad\sum_{m=agn}\frac{\rho_2(abg)(1\star\psi)(m)(1\star\psi)(n)}{\sqrt{amn}}V_{k}\Big(\frac{d_1^2m}{Q}\Big)V_{k}\Big(\frac{d^2bg^2n}{Q}\Big)
\end{align*}
and
\begin{align}\label{IOD}
E_k=q^{-2}\sum_{\substack{d_1,d\\abg\leq X^2}}\frac{\psi(d_1)\psi(d)\mu(g)(1\star\psi)(b)\rho_2(abg)}{d_1d\sqrt{a}bg^{3/2}}\sum_{c\geq1}\frac{T(c)}{c^2},
\end{align}
with
\begin{align}\label{Tcformula}
T(c)=cq\sum_{m,n}\frac{(1\star\psi)(m)(1\star\psi)(n)}{\sqrt{mn}}S(m,agn;cq)J_{1}\Big(\frac{4\pi\sqrt{agmn}}{cq}\Big)V_{k}\Big(\frac{d_1^2m}{Q}\Big)V_{k}\Big(\frac{d^2bg^2n}{Q}\Big).
\end{align}
Proposition \ref{mainprop2} and Proposition \ref{mainprop3} are immediate consequences of the following two results.

\begin{proposition}\label{prop:second moment main term}
Provided that $q,X\geq D^{24}$ 
we have 
\begin{align*}
\mathcal{M}_1^D &=4\big(1+\psi(q)\big)^2\mathfrak{S}_2L'(1,\psi)^2+O_\eps\big(L(1,\psi)(\log q)^{37+\eps}\big)+ O_A((\log q)^{-A})\\ 
&\qquad\qquad+ O_\eps(q^{-1/4+\eps}D^{-1/4}X^{3/4})
\end{align*}
and
\begin{align*}
&\mathcal{M}_2^D = 16\big(1-\psi(q)\big)^2 \mathfrak{S}_2\bigg(1 + O \Big(\frac{\log \log (1/L(1,\psi)\log D)}{\log (1/L(1,\psi)\log D)} \Big) \bigg)\\
&\qquad \times\Big(L''(1,\psi)+(\log Q-2\gamma)L'(1,\psi)\Big)^2+ O_\varepsilon( L(1,\psi)(\log q)^{39+\varepsilon})  +O_\eps(q^{-1/4+\eps}D^{-1/4}X^{3/4}).
\end{align*}
\end{proposition}

\begin{proposition}\label{propE_k}
Provided that $D^8X^6\ll q^{1-\eps}$ we have 
\begin{align*}
E_k&\ll_\eps L(1,\psi)(\log q)^{2k+24+\eps}+ q^{-1/12+\eps}D^{5/3}X^{5/2}+q^{-1/4+\eps}D^{17/4}X^{19/4}.
\end{align*}
\end{proposition}

We shall estimate the error term $E_k$ in Sections \ref{IkODfirst}--\ref{IkODlast}. We finish this section with the initial evaluation of the diagonal contribution $\mathcal{M}_k^{D}$. Inverting \eqref{recursion} and using \eqref{formulaV} we get
\begin{align*}
\mathcal{M}_k^D&=(k!)^2\big(1+(-1)^{k+1}\psi(q)\big)^2\sum_{\substack{d_1,d\\ab\leq X^2}}\frac{\psi(d_1)\psi(d)\rho_2(ab)}{d_1dab}\nonumber\\
&\qquad\qquad\sum_{n}\frac{(1\star\psi)(an)(1\star\psi)(bn)}{n}V_{k}\Big(\frac{d_1^2an}{Q}\Big)V_{k}\Big(\frac{d^2bn}{Q}\Big)\nonumber\\
&=(k!)^2\big(1+(-1)^{k+1}\psi(q)\big)^2\frac{1}{(2\pi i)^2}\int_{(1)}\int_{(1)}G(u)G(v)\Gamma(1+u)^2\Gamma(1+v)^2Q^{u+v}\nonumber\\
&\qquad\qquad  L(1+2u,\psi)L(1+2v,\psi)\mathcal{W}(u,v)\frac{dudv}{u^{k+1}v^{k+1}},
\end{align*}
where
\begin{align*}
\mathcal{W}(u,v)=\sum_{ab\leq X^2}\frac{\rho_2(ab)}{a^{1+u}b^{1+v}}\sum_{n\geq1}\frac{(1\star\psi)(an)(1\star\psi)(bn)}{n^{1+u+v}} .
\end{align*}

We move the lines of integration to $\text{Re}(u) =\text{Re}(v) = \delta = 1/\log X$ without encountering any poles. To proceed further we truncate the sum over $n$ in $\mathcal{W}(u,v)$ to be finite. Setting $T = \exp((\log X)(\log \log X)^2)$, we find by trivial estimation that
\begin{align*}
\mathcal{W}(u,v) &= \sum_{ab\leq X^2}\frac{\rho_2(ab)}{a^{1+u}b^{1+v}}\sum_{n\leq T}\frac{(1\star\psi)(an)(1\star\psi)(bn)}{n^{1+u+v}} + O_A((\log q)^{-A}).
\end{align*}
Showing that the contribution to $\mathcal{W}(u,v)$ from $X < n \leq T$ is negligible is more challenging and requires an in-depth analysis. We do not provide the details here, but all the necessary tools can be found in the proof of Lemma \ref{lemmaB1} below (see also the remark after Lemma \ref{lemmaB1}). Naturally, using the lacunarity of $(1\star\psi)$ is the main idea. We content ourselves here with stating that
\begin{align}\label{eq:bound for truncated S2 u v}
\begin{split}
\mathcal{W}(u,v) &= \mathcal{S}_2(u,v)+O_\eps\big(L(1,\psi)(\log q)^{33+\eps}\big)+ O_A((\log q)^{-A}),
\end{split}
\end{align}
where
\begin{align}\label{eq:def of S2 u v}
\mathcal{S}_2(u,v) &= \sum_{ab\leq X^2}\frac{\rho_2(ab)}{a^{1+u}b^{1+v}}\sum_{n\leq X}\frac{(1\star\psi)(an)(1\star\psi)(bn)}{n^{1+u+v}}.
\end{align}

The contribution from the error terms in \eqref{eq:bound for truncated S2 u v} to $\mathcal{M}_k^D$ is
\begin{align*}
\ll_{\varepsilon,A} L(1,\psi) (\log q)^{35+2k + \varepsilon} + (\log q)^{-A}.
\end{align*}
Observe that we have the bound $\mathcal{S}_2(u,v)\ll_\eps X^{-3\text{Re}(u+v)+\eps}$ for $\text{Re}(u),\text{Re}(v)\ll 1/\log q$. We also have the bounds
\begin{align*}
\frac{\partial^{i+j}}{\partial u^i \partial v^j} \mathcal{S}_2(u,v) \ll_{\varepsilon,i,j} (\log q)^{24+i+j+\varepsilon}
\end{align*}
for $\text{Re}(u),\text{Re}(v)\geq0$ and $i,j\geq 0$. It hence suffices to study
\begin{align*}
(k!)^2\big(1+(-1)^{k+1}\psi(q)\big)^2\frac{1}{(2\pi i)^2}\int_{(\delta)}\int_{(\delta)}G(u)&G(v)\Gamma(1+u)^2\Gamma(1+v)^2Q^{u+v}\nonumber\\
&\times L(1+2u,\psi)L(1+2v,\psi)\mathcal{S}_2(u,v)\frac{dudv}{u^{k+1}v^{k+1}}.
\end{align*}

We move the $v$-contour to  $\text{Re}(v) = -1/4+\varepsilon$, crossing a pole of order $(k+1)$ at $v=0$ with residue
\begin{align*}
R_{1}&=\frac{2\big(1+\psi(q)\big)^2L'(1,\psi)}{2\pi i}\int_{(\delta)}G(u)\Gamma(1+u)^2Q^{u}
L(1+2u,\psi)\mathcal{S}_2(u,0)\frac{du}{u^{2}}\\
&\qquad\qquad+O_\eps\big(L(1,\psi)(\log q)^{28+\eps}\big)
\end{align*}
if $k=1$, and
\begin{align*}
R_{2}&=\frac{8\big(1-\psi(q)\big)^2}{2\pi i}\int_{(\delta)}G(u)\Gamma(1+u)^2Q^{u}
L(1+2u,\psi)\mathcal{S}_2(u,0)\\
&\qquad\qquad\bigg(L''(1,\psi)+\Big(\log Q-2\gamma+\frac{\partial_v\mathcal{S}_2}{\mathcal{S}_2}(u,0)\Big)L'(1,\psi)\bigg)\frac{du}{u^{3}}+O_\eps\big(L(1,\psi)(\log q)^{30+\eps}\big)
\end{align*}
if $k=2$. The contribution of the new integral is bounded trivially by
\[
\ll_\eps Q^{-1/4+\eps}X^{3/4}\ll_\eps q^{-1/4+\eps}D^{-1/4}X^{3/4}.
\]

Regarding $R_{1}$ and $R_{2}$, we move the line of integration to $\text{Re}(u) = -1/4+\eps$, crossing a pole at $u=0$ with residue
\begin{align*}
&4\big(1+\psi(q)\big)^2\mathcal{S}_2(0,0)L'(1,\psi)^2+O_\eps\big(L(1,\psi)(\log q)^{27+\eps}\big)
\end{align*}
if $k=1$, and
\begin{align*}
16(1&-\psi(q))^2\mathcal{S}_2(0,0)\Big(L''(1,\psi)+(\log Q-2\gamma)L'(1,\psi)\Big)^2\\
&+32\big(1-\psi(q)\big)^2\partial_u \mathcal{S}_2(0,0) L'(1,\psi) \Big(L''(1,\psi)+(\log Q-2\gamma)L'(1,\psi)\Big)\\
&+16\big(1-\psi(q)\big)^2\partial_{uv}^2 \mathcal{S}_2(0,0)L'(1,\psi)^2+O_\eps\big(L(1,\psi)(\log q)^{29+\eps}\big).
\end{align*}
if $k=2$, as $\mathcal{S}_2$ is symmetric in $u$ and $v$. The final integrals, like before, are bounded trivially by $O_\eps(q^{-1/4+\eps}D^{-1/4}X^{3/4})$.  We therefore obtain
\begin{equation}\label{eq:intermediate eval of M_1 D}
\begin{split}
\mathcal{M}_1^D&=4\big(1+\psi(q)\big)^2\mathcal{S}_2(0,0)L'(1,\psi)^2+O_\eps\big(L(1,\psi)(\log q)^{37+\eps}\big)+ O_A((\log q)^{-A}) \\ 
&\qquad\qquad+ O_\eps(q^{-1/4+\eps}D^{-1/4}X^{3/4})
\end{split}
\end{equation}
and
\begin{equation}\label{eq:intermediate eval M_2 D}
\begin{split}
\mathcal{M}_2^D &= 16\big(1-\psi(q)\big)^2\mathcal{S}_2(0,0)\Big(L''(1,\psi)+(\log Q-2\gamma)L'(1,\psi)\Big)^2\\
&\qquad\qquad+32\big(1-\psi(q)\big)^2 \partial_u \mathcal{S}_2(0,0) L'(1,\psi)\Big(L''(1,\psi)+(\log Q-2\gamma)L'(1,\psi)\Big)\\
&\qquad\qquad+16\big(1-\psi(q)\big)^2\partial_{uv}^2 \mathcal{S}_2(0,0) L'(1,\psi)^2+O_\eps\big(L(1,\psi)(\log q)^{39+\eps}\big)\\
&\qquad\qquad+O_A((\log q)^{-A})+O_\eps(q^{-1/4+\eps}D^{-1/4}X^{3/4}).
\end{split}
\end{equation}

\section{The second moments - The singular series}\label{appendB}

In this section we prove Proposition \ref{prop:second moment main term}. In view of \eqref{eq:intermediate eval of M_1 D} and \eqref{eq:intermediate eval M_2 D}, we need to study $\mathcal{S}_2$ and some of its derivatives.

Recall from  \eqref{rho2} that
\[
\rho_2(a)=\sum_{d\leq X}\frac1d\sum_{\substack{a_1a_2=a\\a_1,a_2\leq X/d}}\rho_1(da_1)\rho_1(da_2)
\]
and $\rho_2(a)\ll_\eps (\log q)^{4+\eps} \tau_4(a)$.
In this section we shall prove the following two lemmas.

\begin{lemma}\label{lemmaB1}
Provided that $X\geq D^{24}$ we have
\begin{align*}
\mathcal{S}_2(0,0)&=\mathfrak{S}_2+O_\eps\big(L(1,\psi)(\log q)^{33+\eps}\big)+O_A\big((\log q)^{-A}\big),
\end{align*}
where $\mathfrak{S}_2$ is given by \eqref{singularseries2}.
\end{lemma}
\begin{proof}
We wish to deduce an asymptotic formula for the sum
\begin{align*}
\mathcal{S}_2(0,0) &= \sum_{\substack{ab \leq X^2}}\frac{1}{ab}\sum_{d \leq X}\frac{1}{d}\sum_{\substack{m_1m_2 = ab \\ dm_i \leq X}}\rho_1(dm_1)\rho_1(dm_2)\sum_{\substack{n \leq X}}\frac{(1\star\psi)(an)(1\star\psi)(bn)}{n}.
\end{align*}
Here we have written $m_i$ instead of $n_i$ to reduce visual similarity between the variables. We wish to compare $\mathcal{S}_2(0,0)$ with the Euler product
\begin{align*}
\mathcal{P}_2 &:= \prod_{\substack{p \leq X}}\sum_{\substack{a,b,d,m_1,m_2,n \geq 0 \\ m_1+m_2 = a+b}}\frac{\rho_1(p^{d+m_1})\rho_1(p^{d+m_2})(1\star\psi)(p^{a+n})(1\star\psi)(p^{b+n})}{p^{a+b+d+n}}\\
&=\sum_{\substack{p \mid ab \Rightarrow p \leq X}}\frac{1}{ab}\sum_{p \mid d \Rightarrow p \leq X}\frac{1}{d}\sum_{\substack{m_1m_2 = ab}}\rho_1(dm_1)\rho_1(dm_2)\sum_{\substack{p \mid n \Rightarrow p \leq X}}\frac{(1\star\psi)(an)(1\star\psi)(bn)}{n}.
\end{align*}
Our strategy is the same as that in Lemma \ref{lemmaA1}. We first use smooth number estimates to truncate the variables $a,b,d,n$. We then use lacunarity to reduce the size of the variables much further and show that $\mathcal{P}_2$ is equal to $\mathcal{S}_2(0,0)$, up to an acceptably small error. Having done so, we then evaluate $\mathcal{P}_2$.

We begin with the smooth number estimates. Let $T = \exp((\log X)(\log \log X)^2)$. We wish to show that the contribution to $\mathcal{P}_2$ from any variable larger than $T$ is negligible. Let us consider the contribution from $n > T$, say. Set $\delta = 1/\log X$. Then by the triangle inequality the contribution to $\mathcal{P}_2$ from $n > T$ is
\begin{align*}
&\ll_\eps (\log q)^\varepsilon \sum_{\substack{p \mid ab \Rightarrow p \leq X}}\frac{\tau(a)\tau(b)}{ab}\sum_{\substack{m_1m_2 = ab}}\tau(m_1)\tau(m_2)\sum_{p \mid d \Rightarrow p \leq X}\frac{\tau(d)^2}{d}\sum_{\substack{ n > T\\p \mid n \Rightarrow p \leq X }}\frac{\tau(n)^2}{n}\\
&\ll_\eps(\log q)^{4+\varepsilon}\sum_{\substack{p \mid ab \Rightarrow p \leq X}}\frac{\tau(a)\tau(b)}{ab}\sum_{\substack{m_1m_2 = ab}}\tau(m_1)\tau(m_2)\sum_{\substack{n > T\\p \mid n \Rightarrow p \leq X}}\frac{\tau(n)^2}{n}\\
&\ll_\eps(\log q)^{20+\varepsilon}\sum_{\substack{n > T\\p \mid n \Rightarrow p \leq X}}\frac{\tau(n)^2}{n}\ll_\eps T^{-\delta}(\log q)^{20+\varepsilon}\sum_{\substack{p \mid n \Rightarrow p \leq X}}\frac{\tau(n)^2}{n^{1-\delta}}\\
&\ll_\eps T^{-\delta}(\log q)^{20+\varepsilon}\prod_{p \leq X}\left(1 + \frac{4p^{\delta}}{p} \right)\ll_\eps T^{-\delta}(\log q)^{20+\varepsilon}\prod_{p \leq X}\left(1 + \frac{11}{p} \right)\ll_\eps (\log q)^{31+\varepsilon}T^{-\delta},
\end{align*}
where the penultimate inequality follows since $p^{\delta}\leq e$ and $4e < 11$. By the definition of $T$ and $\delta$ we have
\begin{align*}
T^{\delta} = \exp ((\log \log X)^2)\gg_A (\log q)^A,
\end{align*}
so the contribution from $n > T$ to $\mathcal{P}_2$ is $O_A((\log q)^{-A})$. By similar arguments we show that the cost of truncating to $a,b$ or $d \leq T$ is $O_A((\log q)^{-A})$. Hence
\begin{align*}
\mathcal{P}_2 &= \sum_{\substack{a,b \leq T\\p \mid ab \Rightarrow p \leq X}}\frac{1}{ab}\sum_{\substack{d \leq T\\p \mid d \Rightarrow p \leq X}}\frac{1}{d}\sum_{\substack{m_1m_2 = ab}}\rho_1(dm_1)\rho_1(dm_2)\sum_{\substack{n \leq T\\p \mid n \Rightarrow p \leq X}}\frac{(1\star\psi)(an)(1\star\psi)(bn)}{n}\\
&\qquad\qquad+O_A((\log q)^{-A}).
\end{align*}

Now we use lacunarity to cut down the size of the variables much further. These arguments are rather more complicated than those we have heretofore seen, since the variables are quite entangled with each other.

We begin with the $n$ variable. We wish to show that, by lacunarity, the contribution to $\mathcal{P}_2$ from $X < n \leq T$ is negligible; call this contribution $\mathcal{C}_n$. By the triangle inequality
\begin{align*}
\mathcal{C}_n &\ll_\eps(\log q)^\varepsilon\sum_{\substack{p \mid ab \Rightarrow p \leq X}}\frac{1}{ab}\sum_{\substack{m_1m_2 = ab}}\tau(m_1)\tau(m_2)\sum_{\substack{p \mid d \Rightarrow p \leq X}}\frac{\tau(d)^2}{d}\sum_{\substack{X< n \leq T\\p \mid n \Rightarrow p \leq X}}\frac{(1\star\psi)(an)(1\star\psi)(bn)}{n}\\
&\ll_\eps (\log q)^{4+\varepsilon}\sum_{\substack{p \mid ab \Rightarrow p \leq X}}\frac{1}{ab}\sum_{\substack{m_1m_2 = ab}}\tau(m_1)\tau(m_2)\sum_{\substack{X< n \leq T\\p \mid n \Rightarrow p \leq X }}\frac{(1\star\psi)(an)(1\star\psi)(bn)}{n}.
\end{align*}
We may factor $n = n_0 n_+n_-$, where $p \mid n_0$ implies $p\mid D$ and $p \mid n_{\pm}$ implies $\psi(p) = \pm 1$, respectively. Since $n = n_0n_+n_- > X$ we must have that one of the variables $n_0,n_+,n_-$ exceeds $X^{1/3}$. If $n_0 > X^{1/3}$ then we factor $n = n_0 m$, where $n_0 \mid D^\infty$ and $(m,D)=1$ to get
\begin{align*}
\mathcal{C}_n &\ll_\eps(\log q)^{4+\varepsilon}\sum_{\substack{p \mid ab \Rightarrow p \leq X}}\frac{\tau(a)\tau(b)}{ab}\sum_{\substack{m_1m_2 = ab}}\tau(m_1)\tau(m_2)\sum_{\substack{p \mid m\Rightarrow p \leq X \\ (m,D)=1}}\frac{\tau(m)^2}{m}\sum_{\substack{ n_0 > X^{1/3}\\n_0 \mid D^\infty }}\frac{\tau(n_0)^2}{n_0}\\
&\ll_\eps(\log q)^{24+\varepsilon}\sum_{\substack{ n_0 > X^{1/3}\\n_0 \mid D^\infty}}\frac{\tau(n_0)^2}{n_0}\ll_\eps X^{-1/6+\varepsilon}\sum_{n_0 \mid D^\infty}\frac{\tau(n_0)^2}{\sqrt{n_0}}\ll_\eps X^{-1/6+\eps}.
\end{align*}

Next, suppose that the variable $n_+$ is greater than $X^{1/3}$. Then we may factor $n = n_+ m$, where $(n_+,m)=1$ and $n_+ > X^{1/3}$. We then see that
\begin{align*}
\mathcal{C}_n &\ll_\eps(\log q)^{4+\varepsilon}\sum_{\substack{p \mid ab \Rightarrow p \leq X}}\frac{\tau(a)\tau(b)}{ab}\sum_{\substack{m_1m_2 = ab}}\tau(m_1)\tau(m_2)\sum_{\substack{ X^{1/3}< n \leq T\\p \mid n \Rightarrow p \leq X \\ p \mid n \Rightarrow \psi(p)=1 }}\frac{\tau(n)^2}{n}\sum_{\substack{p \mid m \Rightarrow p \leq X \\(m,n)=1}}\frac{\tau(m)^2}{m}\\
&\ll_\eps(\log q)^{24+\varepsilon}\sum_{\substack{ X^{1/3}< n \leq T\\p \mid n \Rightarrow p \leq X \\ p \mid n \Rightarrow \psi(p)=1 }}\frac{\tau(n)^2}{n}.
\end{align*}
If $\psi(p)=1$ for all $p \mid n$ then $\tau(n) = (1\star \psi)(n)$, so
\begin{align*}
\sum_{\substack{ X^{1/3}< n \leq T\\p \mid n \Rightarrow p \leq X \\ p \mid n \Rightarrow \psi(p)=1 }}\frac{\tau(n)^2}{n} &\leq \sum_{\substack{ X^{1/3}< n \leq T\\p \mid n \Rightarrow p \leq X}}\frac{(1\star\psi)(n)\tau(n)}{n} \ll_{\varepsilon,A} L(1,\psi)(\log q)^{5+\varepsilon} + (\log q)^{-A},
\end{align*}
the last inequality following from Lemma \ref{lem:updated lacunarity lemma}.

Lastly, suppose that $n_- > X^{1/3}$. This case is slightly more complicated, because here we need to use the structure of the function $(1\star\psi)(an)(1\star\psi)(bn)$ and because of the fact that $a$ and $b$ are entangled with $n$ is somewhat important. We factor $n_- = n_1n_2n_3n_4$, where
\begin{align*}
p\mid n_1 &\Rightarrow p \nmid ab,\\
p \mid n_2 &\Rightarrow p \mid a, p \nmid b, \\
p \mid n_3 &\Rightarrow p \nmid a, p \mid b,\\
p \mid n_4 &\Rightarrow p \mid a, p \mid b.
\end{align*}
Since $n_- > X^{1/3}$ we see that one of the variables $n_i$ satisfies $n_i > X^{1/12}$. Assume first that $n_1 > X^{1/12}$. Then we may factor $n = n_1 m$, where $(m,n_1) = 1$ to obtain
\begin{align*}
\mathcal{C}_n &\ll_\eps(\log q)^{4+\varepsilon}\sum_{\substack{p \mid ab \Rightarrow p \leq X}}\frac{\tau(a)\tau(b)}{ab}\sum_{\substack{m_1m_2 = ab}}\tau(m_1)\tau(m_2)\sum_{\substack{ X^{1/12} < n_1 \leq T \\p \mid n_1 \Rightarrow \psi(p) = -1\\ (n_1,ab)=1}}\frac{(1\star\psi)(n_1)^2}{n_1}\sum_{\substack{p \mid m \Rightarrow p \leq X \\ (m,n_1)=1}}\frac{\tau(m)^2}{m}\\
&\ll_\eps (\log q)^{24+\varepsilon}\sum_{\substack{ X^{1/12} < n_1 \leq T \\p \mid n_1 \Rightarrow \psi(p) = -1}}\frac{(1\star\psi)(n_1)^2}{n_1}.
\end{align*}
Given the condition on primes dividing $n_1$ we see that $(1\star\psi)(n_1)^2 = \mathbf{1}(n_1 = \square)$, and therefore
\begin{align*}
\mathcal{C}_n &\ll_\eps X^{-1/24+\varepsilon}.
\end{align*}

Now consider the case in which $n_i > X^{1/12}$, where $2 \leq i \leq 4$. Observe that $n_2,n_4 \mid a^\infty$ and $n_3 \mid b^\infty$. Therefore, by the symmetry between $a$ and $b$ we may factor $n=n_a m$, say, where $n_a > X^{1/12}$, whence
\begin{align*}
\mathcal{C}_n &\ll_\eps(\log q)^{4+\varepsilon}\sum_{\substack{p \mid ab \Rightarrow p \leq X}}\frac{\tau(a)\tau(b)}{ab}\sum_{\substack{m_1m_2 = ab}}\tau(m_1)\tau(m_2)\sum_{\substack{n_a > X^{1/12}\\n_a \mid a^\infty }}\frac{\tau(n_a)^2}{n_a}\sum_{p \mid m \Rightarrow p \leq X}\frac{\tau(m)^2}{m}\\
&\ll_\eps q^{\eps}\sum_{\substack{p \mid ab \Rightarrow p \leq X}}\frac{1}{ab}\sum_{\substack{n_a > X^{1/12}\\n_a \mid a^\infty }}\frac{\tau(n_a)^2}{n_a}\ll_\eps X^{-1/24+\eps}\sum_{\substack{p \mid ab \Rightarrow p \leq X}}\frac{1}{ab}\sum_{\substack{n_a \mid a^\infty}}\frac{\tau(n_a)^2}{\sqrt{n_a}}\\
&\ll_\eps X^{-1/24+\eps}\sum_{\substack{p \mid ab \Rightarrow p \leq X}}\frac{\tau(a)}{ab}\ll_\eps X^{-1/24+\eps}.
\end{align*}
The above arguments thus yield
\begin{align*}
\mathcal{P}_2 &= \sum_{\substack{a,b \leq T\\p \mid ab \Rightarrow p \leq X}}\frac{1}{ab}\sum_{\substack{d \leq T\\p \mid d \Rightarrow p \leq X }}\frac{1}{d}\sum_{\substack{m_1m_2 = ab}}\rho_1(dm_1)\rho_1(dm_2)\sum_{\substack{n \leq X}}\frac{(1\star\psi)(an)(1\star\psi)(bn)}{n}\\
&\qquad\qquad+ O_\varepsilon(L(1,\psi)(\log q)^{29+\varepsilon})+O_A((\log q)^{-A}).
\end{align*}

We next wish to impose the conditions $dm_i \leq X$. Note that if $dm_1 \leq X$ and $dm_2 \leq X$ then $ab \leq X^2$. By symmetry it suffices to consider the contribution $\mathcal{C}_d$ from $dm_1 > X$. By the triangle inequality,
\begin{align*}
\mathcal{C}_d &\ll_\eps (\log q)^\varepsilon\sum_{\substack{a,b \leq T\\p \mid ab \Rightarrow p \leq X }}\frac{\tau(a)\tau(b)\tau(ab)}{ab}\sum_{\substack{ d \leq T\\p \mid d \Rightarrow p \leq X}}\frac{\tau(d)}{d}\mathop{\sum}_{\substack{m_1\mid ab\\dm_1 > X }}|\rho_1(dm_1)|\sum_{\substack{n \leq X}}\frac{\tau(n)^2}{n}\\
&\ll_\eps(\log q)^{4+\varepsilon}\sum_{\substack{a,b \leq T\\p \mid ab \Rightarrow p \leq X }}\frac{\tau(a)\tau(b)\tau(ab)}{ab}\sum_{\substack{d \leq T\\p \mid d \Rightarrow p \leq X }}\frac{\tau(d)}{d}\mathop{\sum}_{\substack{m_1\mid ab \\ dm_1 > X}}|\rho_1(dm_1)|.
\end{align*}
We change the order of summation to put the summations over $a$ and $b$ as the innermost sums, and then change variables to obtain
\begin{align*}
\mathcal{C}_d &\ll_\eps(\log q)^{4+\varepsilon}\sum_{\substack{d \leq T, m_1 \leq T^2 \\ dm_1 > X\\p \mid dm_1 \Rightarrow p \leq X}} \frac{\tau(d)|\rho_1(dm_1)|}{d}\sum_{\substack{p \mid t \Rightarrow p \leq X \\ m_1 \mid t}}\frac{\tau(t)}{t}\sum_{\substack{ab = t}}\tau(a)\tau(b)\\
&\ll_\eps(\log q)^{4+\varepsilon}\sum_{\substack{d \leq T, m_1 \leq T^2 \\ dm_1 > X\\p \mid dm_1 \Rightarrow p \leq X}} \frac{\tau(d)|\rho_1(dm_1)|}{d}\sum_{\substack{p \mid t \Rightarrow p \leq X \\ m_1 \mid t}}\frac{\tau(t)\tau_4(t)}{t}\\
&\ll_\eps(\log q)^{12+\varepsilon}\sum_{\substack{d \leq T, m_1 \leq T^2 \\ dm_1 > X\\p \mid dm_1 \Rightarrow p \leq X}} \frac{\tau(d)\tau(m_1)\tau_4(m_1)|\rho_1(dm_1)|}{dm_1}.
\end{align*}
We change variables again, writing $\ell = dm_1$, to derive
\begin{align*}
\mathcal{C}_d &\ll_\eps(\log q)^{12+\varepsilon} \sum_{\substack{X < \ell \leq T^3\\p\mid\ell \Rightarrow p \leq X }}\frac{|\rho_1(\ell)|}{\ell}\sum_{dm_1 = \ell}\tau(d)\tau(m)\tau_4(m_1) = (\log q)^{12+\varepsilon} \sum_{\substack{X < \ell \leq T^3\\p\mid\ell \Rightarrow p \leq X }}\frac{|\rho_1(\ell)|f(\ell)}{\ell},
\end{align*}
where $f(\ell)$ is a multiplicative function supported on cube-free integers with $f(p) = 10$ and $f(p^2) = 49$ (recall that $\rho_1$ is supported on cube-free integers). To handle the factor of $\rho_1(\ell)$ we factor $\ell = \ell_1\ell_2^2$, where $\ell_1$ and $\ell_2$ are squarefree and coprime to each other. We must have $\ell_1 > X^{1/2}$ or $\ell_2 > X^{1/4}$. The contribution from $\ell_2 > X^{1/4}$ is easily seen to be $O(q^{-\varepsilon})$, so we focus on the case when $\ell_1 > X^{1/2}$. In this case we find that
\begin{align*}
\mathcal{C}_d &\ll _\eps(\log q)^{12+\varepsilon} \sum_{\substack{X^{1/2} < \ell_1 \leq T^3\\p \mid \ell_1 \Rightarrow p \leq X }}\frac{\mu^2(\ell_1)(1\star\psi)(\ell_1) \tau(\ell_1)^{\log 10/\log 2}}{\ell_1} +q^{-\varepsilon}\\
& \ll_{\eps,A} L(1,\psi)(\log q)^{33+\varepsilon} + (\log q)^{-A},
\end{align*}
by Lemma \ref{lem:updated lacunarity lemma}. Therefore
\begin{align*}
\mathcal{P}_2 &=\sum_{\substack{ab \leq X^2}}\frac{1}{ab}\sum_{d \leq X}\frac{1}{d}\sum_{\substack{m_1m_2 = ab \\ dm_i \leq X}}\rho_1(dm_1)\rho_1(dm_2)\sum_{\substack{n \leq X}}\frac{(1\star\psi)(an)(1\star\psi)(bn)}{n}\\
&\qquad\qquad + O_\varepsilon(L(1,\psi)(\log q)^{33+\varepsilon})+O_A((\log q)^{-A}),
\end{align*}
or in other words, we have $\mathcal{S}_2(0,0)$ equals $\mathcal{P}_2$ up to an acceptable error.

The local factor for each $p\leq X$ in $\mathcal{P}_2$ is
\begin{align*}
\sum_{\substack{a,b,d,m_1,m_2,n \geq 0 \\ m_1+m_2 = a+b}}\frac{\rho_1(p^{d+m_1})\rho_1(p^{d+m_2})(1\star\psi)(p^{a+n})(1\star\psi)(p^{b+n})}{p^{a+b+d+n}}.
\end{align*}
There are three different cases to consider, depending on the value of $\psi(p)$.

When $\psi(p) = 0$ the local factor is
\begin{align*}
&\sum_{\substack{a,b,d,m_1,m_2,n \geq 0 \\ m_1+m_2 = a+b \\ d+m_i \leq 1}}\frac{\alpha(p^{d+m_1})\mu(p^{d+m_1})\alpha(p^{d+m_2})\mu(p^{d+m_2})}{p^{a+b+d+n}}\\
&\qquad\qquad= \left(1 - \frac{1}{p}\right)^{-1}\sum_{\substack{d,m_1,m_2 \geq 0 \\ d+m_i \leq 1}}\frac{(-1)^{m_1+m_2}\alpha(p^{d+m_1})\alpha(p^{d+m_2})}{p^{d+m_1+m_2}}\sum_{\substack{a,b \geq 0\\a+b=m_1+m_2}}1\\
&\qquad\qquad=\left(1 - \frac{1}{p}\right)^{-1}\sum_{\substack{d,m_1,m_2 \geq 0 \\ d+m_i \leq 1}}\frac{(-1)^{m_1+m_2}\alpha(p^{d+m_1})\alpha(p^{d+m_2})(m_1+m_2+1)}{p^{d+m_1+m_2}}\\
&\qquad\qquad= \left(1 - \frac{1}{p}\right)^{-1}\Big(1 -\frac{4\alpha(p)}{p} + \frac{3\alpha^2(p)}{p^2} +\frac{\alpha^2(p)}{p}\Big)= \left(1 + \frac{1}{p} \right)^{-2}.
\end{align*}

When $\psi(p) = -1$ we break up the sum according to the value of $d$, and then the values of $m_1$ and $m_2$, to see that the local factor is
\begin{align*}
&\sum_{\substack{a,b,d,m_1,m_2,n \geq 0 \\ m_1+m_2 = a+b\\ d+m_i \in \{0,2\} \\ a+n,b+n \text{ even}}}\frac{h(p^{d+m_1})\alpha^2(p^{d+m_1})\psi(p^{(d+m_1)/2})h(p^{d+m_2})\alpha^2(p^{d+m_2})\psi(p^{(d+m_2)/2})}{p^{a+b+d+n}}\\
&\qquad=\left(1 - \frac{1}{p^2}\right)^{-1}-\frac{2h(p)\alpha^2(p)}{p^2}\sum_{n \geq 0}\frac{1 + \frac{1+(-1)^n}{2}}{p^n}+\frac{h^2(p)\alpha^4(p)}{p^4}\sum_{n \geq 0}\frac{2+\frac{1+(-1)^n}{2}}{p^n}\\
&\qquad\qquad\qquad+\frac{h^2(p)\alpha^4(p)}{p^3}\sum_{n \geq 0}\frac{1 + \frac{1+(-1)^n}{2}}{p^n}+\frac{h^2(p)\alpha^4(p)}{p^2}\left(1 - \frac{1}{p^2}\right)^{-1}\\
&\qquad=\left(1 - \frac{1}{p^2}\right)^{-1}-\frac{2h(p)\alpha^2(p)}{p^2}\frac{p(2p+1)}{(p-1)(p+1)}+\frac{h^2(p)\alpha^4(p)}{p^4}\frac{p(3p+2)}{(p-1)(p+1)}\\
&\qquad\qquad\qquad+\frac{h^2(p)\alpha^4(p)}{p^3}\frac{p(2p+1)}{(p-1)(p+1)}+\frac{h^2(p)\alpha^4(p)}{p^2}\left(1 - \frac{1}{p^2}\right)^{-1}\\
&\qquad= \frac{p^2(p+1)^2}{(p^2+p+1)^2}=\Bigg(\left(1 + \frac{1}{p} \right)^2 + \frac{\psi(p)}{p} \Bigg)^{-2}\left(1 - \frac{\psi(p)}{p}\right)^2.
\end{align*}

Lastly, we consider the local factor when $\psi(p) = 1$. The analysis here is just slightly more complicated because $\rho_1(p)$ and $\rho_1(p^2)$ are both nonzero. Here it will be helpful for us to note that for a positive integer $M$ we have
\begin{align*}
&\sum_{\substack{a,b \geq 0\\a+b = M}}(a+n+1)(b+n+1)=\frac{M(M-1)(M+1)}{6}+(n+1)M(M+1)+(n+1)^2(M+1).
\end{align*}
The terms corresponding to $d=0$ are (breaking up $m_1$ and $m_2$ according to whether they are equal or distinct and using symmetries)
\begin{align*}
&\sum_{\substack{a,b,m_1,m_2,n \geq 0 \\ m_1+m_2 = a+b \\ m_i \leq 2}}\frac{\rho_1(p^{m_1})\rho_1(p^{m_2})(a+n+1)(b+n+1)}{p^{m_1+m_2+n}}\\
&\quad=\sum_{n\geq 0}\frac{(n+1)^2}{p^n}+\frac{4h^2(p)\alpha^2(p)}{p^2}\sum_{n\geq 0}\frac{3(n+1)^2+6(n+1)+1}{p^n}\\
&\qquad+\frac{h^2(p)\alpha^4(p)}{p^4}\sum_{n\geq 0}\frac{5(n+1)^2+20(n+1)+10}{p^n}-\frac{4h(p)\alpha(p)}{p}\sum_{n \geq 0}\frac{2(n+1)^2+2(n+1)}{p^n}\\
&\qquad+\frac{2h(p)\alpha^2(p)}{p^2}\sum_{n\geq 0}\frac{3(n+1)^2+6(n+1)+1}{p^n}-\frac{4h^2(p)\alpha^3(p)}{p^3}\sum_{n\geq 0}\frac{4(n+1)^2+12(n+1)+4}{p^n}.
\end{align*}
The terms corresponding to $d=1$ are
\begin{align*}
&\frac{1}{p}\sum_{\substack{a,b,m_1,m_2,n \geq 0 \\ m_1+m_2 = a+b \\ m_i \leq 2}}\frac{\rho_1(p^{m_1+1})\rho_1(p^{m_2+1})(a+n+1)(b+n+1)}{p^{m_1+m_2+n}}\\
&\qquad= \frac{4h^2(p)\alpha^2(p)}{p}\sum_{n\geq 0}\frac{(n+1)^2}{p^n}+\frac{h^2(p)\alpha^4(p)}{p^3}\sum_{n\geq 0}\frac{3(n+1)^2+6(n+1)+1}{p^n}\\
&\qquad\qquad-\frac{4h^2(p)\alpha^3(p)}{p^2}\sum_{n \geq 0}\frac{2(n+1)^2+2(n+1)}{p^n}
\end{align*}
and the term coming from $d=2$ is
\begin{align*}
\frac{h^2(p)\alpha^4(p)}{p^2}\sum_{n\geq 0}\frac{(n+1)^2}{p^n}.
\end{align*}
When we combine the terms from $d=0,d=1$ and $d=2$, we find that their sum is equal to
\begin{align*}
\frac{p^2(p-1)^2}{(p^2+3p+1)^2}=\left(\left(1 + \frac{1}{p} \right)^2 + \frac{\psi(p)}{p} \right)^{-2}\left(1 - \frac{\psi(p)}{p}\right)^2,
\end{align*}
which completes the proof of the lemma.
\end{proof}

\noindent\textbf{Remark.} In order to obtain \eqref{eq:bound for truncated S2 u v} one uses essentially the same argument that we used to bound $\mathcal{C}_n$ in Lemma \ref{lemmaB1}.

\begin{lemma}\label{lemmaB2}
 With $\mathcal{S}_2$ being given by \eqref{eq:def of S2 u v} and $X\geq D^{24}$ we have
\begin{align*}
\partial_u\mathcal{S}_2(0,0)&\ll_{\eps,A} \mathfrak{S}_2\frac{\log\log(1/L(1,\psi)\log D)}{\log(1/L(1,\psi)\log D)}\log D+L(1,\psi)(\log q)^{34+\eps}
\end{align*}
and
\begin{align*}
\partial_{uv}^2\mathcal{S}_2(0,0)&\ll_{\eps,A} \mathfrak{S}_2\frac{\log\log(1/L(1,\psi)\log D)}{\log(1/L(1,\psi)\log D)}(\log D)^2+L(1,\psi)(\log q)^{35+\eps}.
\end{align*}
\end{lemma}
\begin{proof}
The proof is similar to the proof of Lemma \ref{lemmaA2}, but more tedious. We give some details, but leave the patient reader to work out the full argument.

We prove the second statement, since the first is similar and easier. Let $\varepsilon_0 > 0$ be a sufficiently small, fixed constant. We have
\begin{align*}
\partial_{uv}^2 \mathcal{S}_2(0,0) &= \frac{1}{(2\pi i)}\oint\oint \mathcal{S}_2(u,v) \frac{dudv}{u^2v^2},
\end{align*}
where $u$ and $v$ run over the circle centered at the origin of radius $(\log q)^{-1-\varepsilon_0}$. A slight adjustment of the proof of Lemma \ref{lemmaB1} shows that
\begin{align*}
\mathcal{S}_2(u,v) &=\mathcal{P}_2(u,v)+ O_\varepsilon (L(1,\psi)(\log q)^{33+\varepsilon})+O_A ((\log q)^{-A}),
\end{align*}
where
\begin{align*}
\mathcal{P}_2(u,v) &=\prod_{\substack{p \leq X}}\sum_{\substack{a,b,d,m_1,m_2,n \geq 0 \\ m_1+m_2 = a+b}}\frac{\rho_1(p^{d+m_1})\rho_1(p^{d+m_2})(1\star\psi)(p^{a+n})(1\star\psi)(p^{b+n})}{p^{a(1+u)+b(1+v)+d+n(1+u+v)}}.
\end{align*}
By integration we have
\begin{align*}
\frac{1}{2\pi i}\oint \mathcal{P}_2(u,v) \frac{dv}{v^2} = \frac{\partial}{\partial v}\mathcal{P}_2(u,v)\Big|_{v=0} = \mathcal{Q}_2(u) \mathcal{P}_2(u,0),
\end{align*}
say, where
\begin{align*}
\mathcal{Q}_2(u) &= \frac{1}{\mathcal{P}_2(u,0)}\frac{\partial}{\partial v}\mathcal{P}_2(u,v)\Big|_{v=0}
\end{align*}
is a sum over primes $p\leq X$. As $\mathcal{P}_2(u,v) = \mathcal{P}_2(v,u)$ we then deduce that
\begin{align*}
&\frac{1}{(2\pi i)}\oint\oint \mathcal{P}_2(u,v) \frac{dudv}{u^2v^2} = \left(\mathcal{Q}_2'(0) + \mathcal{Q}_2(0)^2\right)\mathcal{P}_2(0,0)= \left(\mathcal{Q}_2'(0) + \mathcal{Q}_2(0)^2\right)\mathfrak{S}_2.
\end{align*}
It then suffices to bound $\mathcal{Q}_2(0)$ and $\mathcal{Q}_2'(0)$. 

To bound these sums over primes we again discriminate according to the value of $\psi(p)$. Since
\begin{align*}
\sum_{p \mid D}\frac{(\log p)^{O(1)}}{p} \ll (\log \log D)^{O(1)}
\end{align*}
and
\begin{align*}
\sum_p \frac{(\log p)^{O(1)}}{p^2} \ll 1,
\end{align*}
it suffices to bound
\begin{align*}
\sum_{\substack{p \leq X \\ \psi(p) = 1}} \frac{(\log p)}{p}
\end{align*}
and
\begin{align*}
\sum_{\substack{p \leq X \\ \psi(p) = 1}} \frac{(\log p)^2}{p}\ll (\log q)\sum_{\substack{p \leq X \\ \psi(p) = 1}} \frac{(\log p)}{p}.
\end{align*}
Now apply Lemma \ref{lem:lac sum over primes} and the lemma follows.
\end{proof}

\begin{proof}[Proof of Proposition \ref{prop:second moment main term}]
We derive the expression for $\mathcal{M}_1^D$ immediately from \eqref{eq:intermediate eval of M_1 D}, the trivial bound $L'(1,\psi) \ll (\log D)^2$ and Lemma \ref{lemmaB1}.

Recall from \eqref{eq:intermediate eval M_2 D} that
\begin{align*}
\mathcal{M}_2^D &= 16\big(1-\psi(q)\big)^2\mathcal{S}_2(0,0)\Big(L''(1,\psi)+(\log Q-2\gamma)L'(1,\psi)\Big)^2\\
&\qquad\qquad+32\big(1-\psi(q)\big)^2 \partial_u \mathcal{S}_2(0,0) L'(1,\psi)\Big(L''(1,\psi)+(\log Q-2\gamma)L'(1,\psi)\Big)\\
&\qquad\qquad+16\big(1-\psi(q)\big)^2\partial_{uv}^2 \mathcal{S}_2(0,0) L'(1,\psi)^2+O_\eps\big(L(1,\psi)(\log q)^{39+\eps}\big)\\
&\qquad\qquad+O_A((\log q)^{-A})+O_\eps(q^{-1/4+\eps}D^{-1/4}X^{3/4}).
\end{align*}
By Lemma \ref{lemmaB1} and the trivial bounds $L'(1,\psi) \ll (\log D)^2, L''(1,\psi) \ll (\log D)^3$, the first term is equal to
\begin{align*}
&16\big(1-\psi(q)\big)^2 \mathfrak{S}_2\Big(L''(1,\psi)+(\log Q-2\gamma)L'(1,\psi)\Big)^2+O_\varepsilon(L(1,\psi)(\log q)^{39+\varepsilon}) + O_A((\log q)^{-A}).
\end{align*}
We use Lemma \ref{lemmaB2} 
to obtain that the second term is
\begin{align*}
&\ll_{\varepsilon,A}\mathfrak{S}_2|L'(1,\psi)| \Big|L''(1,\psi)+(\log Q-2\gamma)L'(1,\psi) \Big| \frac{\log \log(1/L(1,\psi)\log D)}{\log (1/L(1,\psi)\log D)} \log D\\
&\qquad\qquad + L(1,\psi)(\log q)^{39+\varepsilon}.
\end{align*}
We similarly see that the third term is
\begin{align*}
&\ll_{\varepsilon,A} \mathfrak{S}_2 L'(1,\psi)^2  \frac{\log \log(1/L(1,\psi)\log D)}{\log (1/L(1,\psi)\log D)}(\log D)^2 + L(1,\psi)(\log q)^{39+\varepsilon}.
\end{align*}
Using Lemma \ref{lem:deriv expressions} and the trivial bound $\mathfrak{S}_2\leq 1$ to simplify as in the proof of Proposition \ref{mainprop1}, we deduce that
\begin{align*}
&\mathcal{M}_2^D = 16\big(1-\psi(q)\big)^2 \mathfrak{S}_2\bigg(1 + O \Big(\frac{\log \log (1/L(1,\psi)\log D)}{\log (1/L(1,\psi)\log D)} \Big) \bigg)\\
&\qquad\times\Big(L''(1,\psi)+(\log Q-2\gamma)L'(1,\psi)\Big)^2+ O_\varepsilon( L(1,\psi)(\log q)^{39+\varepsilon})+O_\eps(q^{-1/4+\eps}D^{-1/4}X^{3/4}),
\end{align*}
as claimed.
\end{proof}

\section{The second moments - The error term $E_k$}\label{IkODfirst}

In Sections \ref{IkODfirst}--\ref{IkODlast} we shall prove Proposition \ref{propE_k}.


\subsection{Smooth partition of unity}
We apply a dyadic partition of unity to the sum over $m,n$ in $T(c)$ in \eqref{Tcformula}. Let $\omega$ be a smooth non-negative function supported in $[1, 2]$ such that 
\begin{equation*}\label{eq:funcF}
\sum_{M}\omega\Big(\frac xM\Big)=1,
\end{equation*}
where $M$ runs over a sequence of real numbers with $\#\{M: M\leq Z\}\ll \log Z$. Then 
 \begin{equation*} 
T(c)=\sum_{M,N}T_{M,N}(c),
\end{equation*}
where 
 \begin{align}\label{TMNc} 
T_{M,N}(c)&=\frac{cq}{\sqrt{MN}}\sum_{m,n}\,(1\star\psi)(m)(1\star\psi)(n)S(m,agn;cq)\\
&\qquad\qquad J_{1}\Big(\frac{4\pi\sqrt{agmn}}{cq}\Big)V_{k}\Big(\frac{d_1^2m}{Q}\Big)V_{k}\Big(\frac{d^2bg^2n}{Q}\Big)\omega_0\Big(\frac mM\Big)\omega_0\Big(\frac nN\Big)\nonumber
\end{align}
and
\[
\omega_0(x)=x^{-1/2}\omega(x).
\]

Due to the decay of the function $V_k$, we can restrict $M,N$ to $$M,N\ll Q^{1+\varepsilon}$$ with a negligible error term.

\subsection{Removing large values of $c$}

It is convenient to remove large values of $c$. We shall do this by appealing to Proposition 1 in [\textbf{\ref{DFI2}}].

\begin{lemma}\label{lem:removingC}
For any $C\geq q^{-1}\sqrt{MN}X$ we have
 \begin{equation*}
\sum_{c \geq C}\frac{T_{M,N}(c)}{c^2} \ll_{\varepsilon} (Cq)^{\varepsilon}q^{3/2}D(ag)^{3/4}\Big( \frac{\sqrt{MN}}{C}\Big)^{1/2}.
\end{equation*}
\end{lemma}
\begin{proof}
Consider the summation over a dyadic interval, 
 \begin{align*}
\sum_{C\leq c \leq 2C}\frac{T_{M,N}(c)}{c^2}&=\frac{q}{\sqrt{MN}}\sum_{C\leq c \leq 2C}\sum_{m,n}\,(1\star\psi)(m)(1\star\psi)(n)\frac{S(m,agn;cq)}{c}\\
&\qquad\qquad J_{1}\Big(\frac{4\pi\sqrt{agmn}}{cq}\Big)V_{k}\Big(\frac{d_1^2m}{Q}\Big)V_{k}\Big(\frac{d^2bg^2n}{Q}\Big)\omega_0\Big(\frac mM\Big)\omega_0\Big(\frac nN\Big).
\end{align*}
By [\textbf{\ref{DFI2}}; Proposition 1] this is bounded by
\begin{align*}
&\ll_\varepsilon \frac{(Cq)^{\varepsilon}q}{\sqrt{MN}}\Big( \frac{\sqrt{agMN}}{Cq}\Big)^{1/2}(q+M)^{1/2} ( q+agN)^{1/2}\sqrt{MN}\\
& \ll_\varepsilon (Cq)^{\varepsilon}q^{3/2}D(ag)^{3/4}\Big( \frac{\sqrt{MN}}{C}\Big)^{1/2}.
\end{align*}
Summing over all the dyadic intervals we obtain the result.
\end{proof}

The above lemma implies that the contribution of the terms $c\geq C$ to $E_k$ is
\begin{equation}\label{boundtruncationc}
\ll_{\varepsilon} (Cq)
^\eps DX^{5/2}\Big( \frac{\sqrt{MN}}{Cq}\Big)^{1/2}.
\end{equation}
The value of $C$ is chosen by making this estimate equal to the first summand in \eqref{1stboundMkOOD}, namely
\[
DX^{5/2}\Big( \frac{\sqrt{MN}}{Cq}\Big)^{1/2}=CD^{9/4}M^{-1/2}N^{1/4}X^{5/2},
\]
which gives
\begin{align}\label{valueofC}
C=q^{-1/3}D^{-5/6}\sqrt{M}
\end{align}
(note that the entire sum on
$c$ is dropped when $M<q^{2/3}D^{5/3}$). To make sure the condition $C\geq q^{-1}\sqrt{MN}X$ is satisfied we shall need
\[
D^8X^6\ll q^{1-\eps}.
\]
Performing the dyadic summation over $M$ and $N$ we therefore find that the sum over $c$ in $E_{k}$ can be truncated to $1\leq c<C$ at the cost of an error term of size
\begin{equation}\label{boundtruncationc2}
\ll_{\varepsilon} q^\eps DX^{5/2}\big( q^{-2/3}D^{5/6}\sqrt{Q}\big)^{1/2}\ll_{\varepsilon} q^{-1/12+\eps} D^{5/3}X^{5/2}.
\end{equation}
%

\subsection{Voronoi summation formulas}

We need the following Voronoi summation formulas. They are both special cases of Proposition 3.3 in [\textbf{\ref{CI2}}] (see also [\textbf{\ref{CI}}; (15.1)] for the first one).

\begin{lemma}\label{voronoi}
Let $(a,c)=1$ and let $\psi = \psi_{1} \psi_{2}$, where $\psi_{1}$ and $\psi_{2}$ are real characters modulo $D_{1}=(c,D)$ and $D_{2}=D/D_{1}$, respectively. Then for any smooth function $g$ compactly supported in $\mathbb{R}_{>0}$ we have
\begin{align*}
\sum_n (1\star\psi)(n) e \left(\frac{an}{c} \right) g(n) &= \rho_1(a,c) L(1,\psi) \int_0^\infty g(x) dx + T_1(a,c),
\end{align*}
where
\begin{align*}
\rho_1(a,c) &=\frac{1}{c} \Big(\psi(c)+\mathbf{1}_{D|c}\tau(\psi)\psi(a)\Big)
\end{align*}
and
\begin{align*}
T_1(a,c) &= \frac{2\pi i \psi_{1}(a)\psi_{2}(c)}{c\sqrt{D_{2}}}\sum_m (\psi_{1} \star \psi_{2})(m) e \Big(-\frac{\overline{a D_{2}}m}{c} \Big)\int_0^\infty J_0 \Big(\frac{4\pi\sqrt{mx}}{c\sqrt{D_{2}}} \Big)g(x) dx.
\end{align*}
\end{lemma}

\begin{lemma}\label{voronoi2}
Let $(a,\ell)=1$ and let $\psi =\psi_1\psi_2$, where $\psi_1,\psi_2$ are real characters modulo $D_{1},D_{2}$. Let
\[
D_1^\flat=(D_1,\ell), \qquad D_1'=D_1/D_1^\flat,\qquad D_2^\flat=(D_2,\ell), \qquad D_2'=D_2/D_2^\flat
\]
and let $\psi=\psi_{D_1^\flat D_2^\flat}\psi_{D_1'D_2'}=\psi_{D_1^\flat D_2'}\psi_{D_1'D_2^\flat}$ be the corresponding real characters.
 Then for any smooth function $g$ compactly supported in $\mathbb{R}_{>0}$ we have
\begin{align*}
\sum_n (\psi_1\star\psi_2)(n) e \left(\frac{an}{\ell} \right) g(n) &= \rho_2(a,\ell) L(1,\psi) \int_0^\infty g(x) dx + T_2(a,\ell),
\end{align*}
where
\begin{align*}
\rho_2(a,\ell) &= \frac{1}{\ell}\bigg(\tau(\psi_1)\psi_1(a)\psi_2\Big(\frac{\ell}{D_1}\Big)+\tau(\psi_2)\psi_2(a)\psi_1\Big(\frac{\ell}{D_2}\Big)\bigg)
\end{align*}
and
\begin{align*}
T_2(a,\ell) &= \frac{2\pi \epsilon_\ell\psi_{D_1^\flat D_2^\flat}(a)}{\ell\sqrt{D_1'D_2'}}\sum_m (\psi_{D_1^\flat D_2'} \star \psi_{D_1'D_2^\flat})(m) e \Big(-\frac{\overline{a D_1'D_2'}m}{\ell} \Big)\int_0^\infty J_0 \Big(\frac{4\pi\sqrt{mx}}{\ell\sqrt{D_1'D_2'}} \Big)g(x) dx,
\end{align*}
for some $\epsilon_\ell$ independent on $a$ and $|\epsilon_\ell|=1$.
\end{lemma}

\subsection{Initial manipulations}

Let
\begin{equation*}
g(m,n):=V_{k}\Big(\frac{d_1^2m}{Q}\Big)V_{k}\Big(\frac{d^2bg^2n}{Q}\Big)\omega_0\Big(\frac mM\Big)\omega_0\Big(\frac nN\Big).
\end{equation*}
Opening the Kloosterman sum in \eqref{TMNc} and applying Lemma \ref{voronoi} to the sum on $m$ we obtain
\begin{equation*}
T_{M,N}(c)=T_{M,N}^{*}(c)+T_{M,N}^{-}(c),
\end{equation*}
where
\begin{align*}
T_{M,N}^{*}(c)&=\frac{L(1,\psi)}{\sqrt{MN}}\sum_{n}\,(1\star\psi)(n)\Big(\psi(cq)S(agn,0;cq)+\mathbf{1}_{D|c}\tau(\psi)S_\psi(agn,0;cq)\Big)G^*(n)
\end{align*}
and
\begin{equation*}
T_{M,N}^{-}(c)=\frac{2\pi i\psi_1(-D_2)\psi_2(cq)}{\sqrt{D_{2}MN}}\sum_{m',n} (\psi_{1} \star \psi_{2})(m')(1\star\psi)(n)S_{\psi_{1}}(m'- agD_{2}n,0;cq) G^-(m',n),
\end{equation*}
with $D_{1}=(c,D)$, $D_{2}=D/D_{1}$ and $\psi=\psi_{1}\psi_{2}$ being the corresponding real characters. Here
\begin{align*}
G^*(n)&=\int_0^\infty J_{1}\Big(\frac{4\pi\sqrt{agxn}}{cq}\Big)g(x,n) dx,\\
G^-(m',n)&= \int_0^\infty J_0 \Big(\frac{4\pi\sqrt{m' x}}{cq\sqrt{D_{2}}} \Big)J_{1}\Big(\frac{4\pi\sqrt{agxn}}{cq}\Big)g(x,n) dx.
\end{align*}

The contribution of the term $T_{M,N}^{*}(c)$ to $E_k$ is easy to handle. Using Lemma \ref{Nelsonlemma} and the bound $G^*(n)\ll M$, that is bounded by 
\begin{align*}
\ll_\eps L(1,\psi)q^{-2+\eps}X\sqrt{MN}\ll _\eps q^{-1+\eps}DX.
\end{align*}

For $T_{M,N}^{-}(c)$ we write
\begin{align}\label{sumoverh}
T_{M,N}^{-}(c)&=\sum_{h\in\mathbb{Z}}S_{\psi_{1}}(h,0;cq)T_{M,N,h}^{-}(c)\nonumber\\
&=\sum_{h=0}+\sum_{h\ne0}=T_{M,N}^{OD}(c)+T_{M,N}^{OOD}(c),
\end{align}
say, where
\begin{equation}\label{T-MNh}
T_{M,N,h}^{-}(c)=\frac{2\pi i\psi_1(-D_2)\psi_2(cq)}{\sqrt{D_{2}MN}}\sum_{m'- agD_{2}n=h} (\psi_{1} \star \psi_{2})(m')(1\star\psi)(n)G^-(m',n).
\end{equation}
We denote the contributions of $T_{M,N}^{OD}(c)$ and $T_{M,N}^{OOD}(c)$ by $\mathcal{M}_k^{OD}$ and $\mathcal{M}_k^{OOD}$, respectively, and hence
\begin{align}\label{firstestimate}
E_k=\mathcal{M}_k^{OD}+\mathcal{M}_k^{OOD}+O_\varepsilon\big(q^{-1/12+\eps} D^{5/3}X^{5/2}\big),
\end{align}
by \eqref{boundtruncationc2}. 
Note that the sums over $c$ in these terms are restricted to $1\leq c<C$.

%

\section{The off-diagonal $\mathcal{M}_k^{OD}$}\label{sectionMkOD}

We have
\begin{equation*}
T_{M,N}^{OD}(c)=\frac{2\pi i\psi_1(-D_2)\psi_2(cq)S_{\psi_1}(0,0;cq)}{\sqrt{D_2MN}}\sum_{n} (\psi_1 \star \psi_2)(agD_2n)(1\star\psi)(n)G^{-}(agD_2n,n).
\end{equation*}
Note that 
\begin{equation}\label{specialcase}
S_{\psi_1}(0,0;cq)=\begin{cases}
\varphi(cq) & \text{if }D_1=1,\\
0 & \text{if }D_1>1,
\end{cases}
\end{equation}
and so
\begin{align*}
T_{M,N}^{OD}(c)&=\frac{2\pi i\psi(cq)\varphi(cq)}{\sqrt{DMN}}\sum_{n} (1 \star \psi)(agn)(1\star\psi)(n)G^{-}(agDn,n)\\
&=\frac{2\pi i\psi(cq)\varphi(cq)}{\sqrt{DMN}}\sum_{n}(1 \star \psi)(agn)(1\star\psi)(n)V_{k}\Big(\frac{d^2bg^2n}{Q}\Big)\omega_0\Big(\frac nN\Big)\\
&\qquad\qquad\int_0^\infty J_0 \Big(\frac{4\pi\sqrt{agn x}}{cq} \Big)J_{1}\Big(\frac{4\pi\sqrt{agnx}}{cq}\Big)V_{k}\Big(\frac{d_1^2x}{Q}\Big) \omega_0\Big(\frac xM\Big)dx.\\
\end{align*}

We now put this into \eqref{IOD}. At this step we would like to remove the restriction $c< C$. Using the bound $J_0 (x)J_{1}(x)\ll x$ this can be done at the cost of an error term of size
\begin{align*}
&\ll_\eps q^{-2+\eps}D^{-1/2}\frac{MN}{C}X^2\ll_\eps q^{-1/6+\eps}D^{11/6}X^2,
\end{align*}
by \eqref{valueofC}. Once the restriction on $c$ is removed we execute the dyadic summation over $M$ and $N$. After using \eqref{recursion} and a change of variables we obtain
\begin{align*}
\mathcal{M}_k^{OD}&= \frac{i}{q\sqrt{D}}\sum_{\substack{d\\ab\leq X^2}}\frac{\psi(d)\rho_2(ab)}{dab}\sum_{n}\frac{(1 \star \psi)(an)(1\star\psi)(bn)}{n}V_{k}\Big(\frac{d^2bn}{Q}\Big)S(n)\\
&\qquad\qquad+O_\eps\big(q^{-1/6+\eps}D^{11/6}X^2\big),
\end{align*}
where
\[
S(n)=\sum_{d_1}\frac{\psi(d_1)}{d_1}\sum_{c\geq1}\frac{\psi(cq)\varphi(cq)}{c}\int_0^\infty J_0 (x)J_{1}(x)V_{k}\Big(\frac{d_1^2c^2q^2x^2}{16\pi^2anQ}\Big) dx.
\]

We use \eqref{formulaV} to obtain
\begin{align*}
S(n)&=\frac{1}{2\pi i}\int_{(1-\eps)}G(u)\Gamma(1+u)^2\Big(\frac{16\pi^2anQ}{q^2}\Big)^uL(1+2u,\psi)\\
&\qquad\qquad\qquad\bigg(\sum_{c\geq1}\frac{\psi(cq)\varphi(cq)}{c^{1+2u}}\bigg)\bigg(\int_0^\infty J_0 (x)J_{1}(x)x^{-2u}dx\bigg)\frac{du}{u^{k+1}}.
\end{align*}
The sum on $c$ is
\begin{equation}\label{sumoverc}
\psi(q)q\frac{L(2u,\psi)}{L(1+2u,\psi)}+O(1).
\end{equation}
On the other hand, we also have [\textbf{\ref{GR}}; 6.574(2)]
\begin{equation}\label{integraloverx}
\int_0^\infty J_0 (x)J_{1}(x)x^{-2u}dx=\frac{\Gamma(1/2+u)\Gamma(1-u)}{2\sqrt{\pi}\Gamma(1+u)^2},
\end{equation}
and so
\begin{align*}
S(n)&= \frac{\psi(q)q}{4\pi\sqrt{\pi}i}\int_{(1-\eps)}G(u)\Gamma\Big(\frac12+u\Big)\Gamma(1-u)\Big(\frac{16\pi^2anQ}{q^2}\Big)^uL(2u,\psi)\frac{du}{u^{k+1}}+O_\eps\big(q^{\eps}D^2X^2\big).
\end{align*}
It follows that
\begin{align}\label{567}
\mathcal{M}_k^{OD}&= \frac{\psi(q)}{4\pi\sqrt{\pi D}}\sum_{\substack{d\\ab\leq X^2}}\frac{\psi(d)\rho_2(ab)}{dab}\sum_{n}\frac{(1 \star \psi)(an)(1\star\psi)(bn)}{n}V_{k}\Big(\frac{d^2bn}{Q}\Big)\\
&\qquad\quad\int_{(1-\eps)}G(u)\Gamma\Big(\frac12+u\Big)\Gamma(1-u)\Big(\frac{16\pi^2anQ}{q^2}\Big)^uL(2u,\psi)\frac{du}{u^{k+1}}+O_\eps\big(q^{-1/6+\eps}D^{11/6}X^2\big).\nonumber
\end{align}

We now show that the first term in \eqref{567} is small for every fixed $k\in\mathbb{N}$, provided that $L(1,\psi)$ is small.
By multiplicativity we have
\begin{equation*}
\sum_{\substack{n\geq 1}} \frac{(1\star\psi)(an)(1\star\psi)(bn)}{n^s}=f_{a,b}(s)\prod_{p}\bigg(1+\sum_{j\geq1}\frac{(1\star\psi)(p^j)^2}{p^{js}}\bigg)
\end{equation*}
for $\text{Re}(s)>1$, where
\begin{align*}
f_{a,b}(s)&=\prod_{\substack{p| ab}}\bigg(1+\sum_{j\geq1}\frac{(1\star\psi)(p^j)^2}{p^{js}}\bigg)^{-1}\prod_{\substack{p^{a_p}||a\\p\nmid b}}\bigg((1\star\psi)(p^{a_p})+\sum_{j\geq1}\frac{(1\star\psi)(p^{a_p+j})(1\star\psi)(p^j)}{p^{js}}\bigg)\nonumber\\
&\qquad\qquad\prod_{\substack{p^{b_p}||b\\p\nmid a}}\bigg((1\star\psi)(p^{b_p})+\sum_{j\geq1}\frac{(1\star\psi)(p^{j})(1\star\psi)(p^{b_p+j})}{p^{js}}\bigg)\\
&\qquad\qquad\prod_{\substack{p^{a_p}||a\\p^{b_p}||b}}\bigg((1\star\psi)(p^{a_p})(1\star\psi)(p^{b_p})+\sum_{j\geq1}\frac{(1\star\psi)(p^{a_p+j})(1\star\psi)(p^{b_p+j})}{p^{js}}\bigg).\nonumber
\end{align*}
We also have
[\textbf{\ref{BPZ}}; p.12]
\begin{align*}
&\prod_{p}\bigg(1+\sum_{j\geq1}\frac{(1\star\psi)(p^j)^2}{p^{js}}\bigg)\nonumber\\
&\qquad\qquad=\prod_{\substack{p|D}}\bigg(1-\frac{1}{p^s}\bigg)^{-1}\prod_{\substack{\psi(p)=-1}}\bigg(1-\frac{1}{p^{2s}}\bigg)^{-1}\prod_{\substack{\psi(p)=1}}\bigg(1+\frac{1}{p^s}\bigg)\bigg(1-\frac{1}{p^s}\bigg)^{-3},
\end{align*}
and hence
\begin{align*}
\prod_p\bigg(1+\sum_{j\geq1}\frac{(1\star\psi)(p^j)^2}{p^{js}}\bigg)&=\frac{\alpha_D(s)L(s,\psi)^2\zeta(s)^2}{\zeta(2s)},
\end{align*}
where
\[
\alpha_D(s)=\prod_{p|D}\bigg(1+\frac{1}{p^{s}}\bigg)^{-1}.
\]
Thus from \eqref{formulaV} the first term in \eqref{567} is
\begin{align*}
& \frac{\psi(q)}{8\pi^2i\sqrt{\pi D}}\sum_{\substack{ab\leq X^2}}\frac{\rho_2(ab)}{ab}\int_{(1)}\int_{(1-\eps)}G(u)G(v)\Gamma\Big(\frac12+u\Big)\Gamma(1-u)\Gamma(1+v)^2\Big(\frac{16\pi^2aQ}{q^2}\Big)^u\Big(\frac {Q}{b}\Big)^v\nonumber\\
&\qquad\qquad L(2u,\psi)L(1+2v,\psi)\sum_{n\geq1}\frac{(1\star\psi)(an)(1\star\psi)(bn)}{n^{1-u+v}}\frac{dudv}{u^{k+1}v^{k+1}}\\
&\ = \frac{\psi(q)}{8\pi^2i\sqrt{\pi D}}\sum_{\substack{ab\leq X^2}}\frac{\rho_2(ab)}{ab}\int_{(1)}\int_{(1-\eps)}G(u)G(v)\Gamma\Big(\frac12+u\Big)\Gamma(1-u)\Gamma(1+v)^2\Big(\frac{16\pi^2aQ}{q^2}\Big)^u\Big(\frac {Q}{b}\Big)^v\nonumber\\
&\qquad\qquad L(2u,\psi)L(1+2v,\psi)L(1-u+v,\psi)^2\zeta(1-u+v)^2\frac{(\alpha_Df_{a,b})(1-u+v)dudv}{u^{k+1}v^{k+1}\zeta\big(2(1-u+v)\big)}.
\end{align*}

We move the $v$-contour to $\text{Re}(v)=1/2+\eps$, crossing a double pole at $v=u$, and the new integral is trivially bounded by
\[
\ll_\eps q^{-1/2+\eps}D^{3/2}X^2,
\]
by using the bounds $\alpha_D\ll_\eps D^\eps$ and $f_{a,b}\ll_\eps (ab)^\eps$ along these contours. For the residue at $v=u$, we move the line of integration to $\text{Re}(u)=1/\log q$ and see that the contribution of the residue is bounded by
\[
\ll_\eps L(1,\psi)^2(\log q)^{2k+23+\eps}+L(1,\psi)L'(1,\psi)(\log q)^{2k+22+\eps}\ll_\eps L(1,\psi)(\log q)^{2k+24+\eps},
\]
as $f_{a,b}(1)\ll \tau(a)\tau(b)$.
So
\begin{equation}\label{500}
\mathcal{M}_k^{OD}\ll_\eps L(1,\psi)(\log q)^{2k+24+\eps}+q^{-1/6+\eps}D^{11/6}X^2.
\end{equation}

\section{Evaluating $T_{M,N,h}^{-}(c)$}

\subsection{Shifted convolution sum}

We first study the shifted convolution sum
\[
\sum_{\substack{am- bn=h}}(\psi_1\star\psi_2)(m)(1\star\psi)(n)
\]
over dyadic intervals. 

\begin{proposition}\label{propoff}
Let $h\in\mathbb{Z}\backslash\{0\}$, $M,N,P\geq1$. Suppose $\psi_1,\psi_2$ are real characters modulo $D_1, D_2$ and $\psi=\psi_1\psi_2$. 
Let $f$ be a smooth function supported on $\mathbb{R}^+\times\mathbb{R}^+$ such that
\[
m^in^jf^{(ij)}(m,n)\ll \Big(1+\frac mM\Big)^{-A}\Big(1+\frac nN\Big)^{-A}P^{i+j}
\] for any fixed $A,i,j\geq0$. Let
\begin{equation}\label{555}
\mathcal S_{a,b,h} =\sum_{\substack{am-bn=h}}(\psi_1\star\psi_2)(m)(1\star\psi)(n)f(m,n).
\end{equation}
Then
\begin{align*}
\mathcal S_{a,b,h}&=L(1,\psi)^2\mathfrak{S}_{a,b}(h)\frac{1}{ab}\int f\Big(\frac{y+h}{a},\frac{y}{b}\Big) dy\\
&\qquad\qquad+O_\eps\big((abhDMN)^\eps DP^{5/4}(abMN)^{1/4}(aM+bN)^{1/4}\big) ,
\end{align*}
where
\begin{align}\label{mathfrakS1}
\mathfrak{S}_{a,b}(h)&=\sum_{\ell\geq 1}\frac{1}{\ell_a'\ell_b'}\bigg(\tau(\psi_1)\psi(\ell_b')\psi_1(-a')\psi_2\Big(\frac{\ell_a'}{D_1}\Big)S_{\psi_1}(h,0;\ell)\nonumber\\
&\qquad\qquad\qquad\qquad+\tau(\psi_2)\psi(\ell_b')\psi_2(-a')\psi_1\Big(\frac{\ell_a'}{D_2}\Big)S_{\psi_2}(h,0;\ell)\nonumber\\
&\qquad\qquad\qquad\qquad+\mathbf{1}_{D|\ell_b'}\tau(\psi)\tau(\psi_1)\psi(b')\psi_1(-a')\psi_2\Big(\frac{\ell_a'}{D_1}\Big)S_{\psi\psi_1}(h,0;\ell)\nonumber\\
&\qquad\qquad\qquad\qquad+\mathbf{1}_{D|\ell_b'}\tau(\psi)\tau(\psi_2)\psi(b')\psi_2(-a')\psi_1\Big(\frac{\ell_a'}{D_2}\Big)S_{\psi\psi_2}(h,0;\ell)\bigg).
\end{align}
Here $a ' =a/(a, \ell)$, $b ' =b/(b, \ell)$ and $\ell_a' = \ell/(a, \ell)$, $\ell_b' = \ell/(b, \ell)$.

\end{proposition}


 \begin{proof}
We use the delta method, as developed by Duke, Friedlander and Iwaniec in [\textbf{\ref{DFI}}]. 
As  usual, let
$\delta(0)=1$ and $\delta(n)=0$ for $n\ne 0$. Let $L\leq P^{-1/2}\min\{\sqrt{aM},\sqrt{bN}\}$. Then
\begin{equation}\label{deltafnc}
\delta(n) = \sum_{\ell\geq 1}\, \sideset{}{^*}\sum_{k(\text{mod}\ \ell)} e\Big(\frac{kn}{\ell}\Big) \Delta_\ell(n),
\end{equation}
where $\Delta_\ell(u)$ is some smooth function that vanishes if $|u| \leq U=L^2$ and $\ell \geq 2L$ (see [\textbf{\ref{DFI}}; Section 4]) and satisfies (see [\textbf{\ref{DFI}}; Lemma 2])
\begin{equation}\label{bdDelta}
\Delta_\ell(u)\ll \big(\ell L+L^2\big)^{-1}+\big(\ell L+|u|\big)^{-1}.
\end{equation}
We also attach to both sides of \eqref{deltafnc} a redundant factor $\phi(n)$, where $\phi(u)$ is a smooth function supported on $|u|<U$ satisfying $\phi(0)=1$ and $\phi^{(j)}(u)\ll U^{-j}$ for any fixed $j\geq0$. Applying this to \eqref{555} we see that
\begin{align}\label{summn}
\mathcal S_{a,b,h} =\sum_{\ell<2L}\ \sideset{}{^*}\sum_{k(\text{mod}\ \ell)}e\Big(\frac{-hk}{\ell}\Big)\sum_{m,n}(\psi_1\star\psi_2)(m)(1\star\psi)(n)e\Big(\frac{k(am-bn)}{\ell}\Big)F(m,n),
\end{align}
where
\[
F(m,n)=\Delta_\ell(am- bn-h)\phi(am-bn-h)f(m,n).
\]

Let $a ' =a/(a, \ell)$, $b ' =b/(b, \ell)$ and $\ell_a' = \ell/(a, \ell)$, $\ell_b' = \ell/(b, \ell)$. Let 
\[
D_1^\flat=(D_1,\ell_a'), \qquad D_1'=D_1/D_1^\flat,\qquad D_2^\flat=(D_2,\ell_a'), \qquad D_2'=D_2/D_2^\flat,
\]
\[
D_1^*=(\ell_b',D),\qquad D_2^*=D/D_1^*
\]
and let $\psi=\psi_{D_1^\flat D_2^\flat}\psi_{D_1'D_2'}=\psi_{D_1^\flat D_2'}\psi_{D_1'D_2^\flat}=\psi_{1}^*\psi_{2}^*$ be the corresponding real characters. We apply the Voronoi summation
formulas in Lemma \ref{voronoi} and Lemma \ref{voronoi2} to the sums over $n$ and $m$ in \eqref{summn}. In doing so, we can write $\mathcal{S}_{a,b,h}$ as a
principal term plus three error terms,
\begin{equation}\label{afterVoronoi}
\mathcal{S}_{a,b,h}=\mathcal{M}_{a,b,h}+\mathcal{E}_{1;a,b,h}+\mathcal{E}_{2;a,b,h}+\mathcal{E}_{3;a,b,h}.
\end{equation}

We first deal with the error terms. All the three error terms can be treated similarly, so we only focus here on one of them, 
\begin{align*}
\mathcal{E}_{1;a,b,h}&=\sum_{\ell<2L}\ \sideset{}{^*}\sum_{k(\text{mod}\ \ell)}\psi_{D_1^\flat D_2^\flat}(k)\psi_{1}^*(k)e\Big(\frac{-hk}{\ell}\Big)\\
&\qquad \sum_{m',n'}\big(\psi_{D_1^\flat D_2'}\star\psi_{D_1'D_2^\flat}\big)(m')\big(\psi_{1}^*\star\psi_{2}^*\big)(n')e\Big(-\frac{\overline{ka'D_1'D_2'}
m'}{\ell_a'}+\frac{\overline{kb'D_{2}^*}n'}{\ell_b'}\Big)H(m',n'),
\end{align*}
where
\[
H(m',n')=\frac{4\pi^2\epsilon_{a,b,\ell}}{\ell_a'\ell_b'\sqrt{D_1'D_2'D_2^*}}\int_{0}^{\infty}\int_{0}^{\infty} J_0\Big(\frac{4\pi\sqrt{m'x}}{\ell_a'\sqrt{D_{1}'D_2'}}\Big)J_0\Big(\frac{4\pi\sqrt{n'y}}{\ell_b'\sqrt{D_{2}^*}}\Big)F(x,y)dxdy
\]
for some $|\epsilon_{a,b,\ell}|=1$. The sum over $k$ is $S_{\psi_{D_1^\flat D_2^\flat}\psi_{1}^*}(-h,*;\ell)$, which is bounded by
\[
\ll_\eps (h,\ell)^{1/2}\ell^{1/2+\eps},
\]
by the Weil bound [\textbf{\ref{BC}}; Lemma 3]. Hence
\begin{align*}
\mathcal{E}_{1;a,b,h}&\ll_\eps \sum_{\ell<2L}\, (h,\ell)^{1/2}\ell^{1/2+\eps}\sum_{m',n'}\tau(m')\tau(n')\big|H(m',n')\big|.
\end{align*} 

We have 
\[
F^{(ij)}\ll_{i,j} \frac{1}{\ell L}\Big(\frac{ab}{\ell L}\Big)^{i+j}
\]
for any fixed $i,j\geq0$. Using the recurrence formula $(x^\nu J_\nu(x))'=x^\nu J_{\nu-1}(x)$, integration by parts then implies that the sums are negligible unless 
\begin{equation}\label{rangemn}
m'\ll_\eps (abDMN)^\eps\, \frac{a'^2MD_{1}'D_2'}{L^{2}}\qquad\text{and}\qquad n'\ll_\eps (abDMN)^\eps\,\frac{b'^2ND_{2}^*}{L^{2}}.
\end{equation}
For $m',n'$ in these ranges we bound $H(m',n')$ trivially using $J_0(x)\ll x^{-1/2}$ and get
\begin{align*}
H(m',n')&\ll \frac{1}{(\ell_a'\ell_b')^{1/2}(D_1'D_2'D_2^*m'n'MN)^{1/4}}\\
&\qquad\qquad\int\int\bigg|\Delta_\ell(ax-by-h)\phi(ax-by-h)f(x,y)\bigg| dxdy\nonumber\\
&\ll_\eps (abMN)^\eps \frac{1}{ab(\ell_a'\ell_b')^{1/2}(D_1'D_2'D_2^*m'n'MN)^{1/4}}\min\{aM,bN\}\int\big|\Delta_\ell(u)\big| du\nonumber\\
&\ll_\eps (abMN)^\eps \frac{(MN)^{3/4}}{(\ell_a'\ell_b')^{1/2}(D_1'D_2'D_2^*m'n')^{1/4}(aM+bN)},
\end{align*}
by \eqref{bdDelta}. So summing over $m',n'$ in the range \eqref{rangemn} we obtain
\[
\sum_{m',n'}\tau(m')\tau(n')\big|H(m',n')\big|\ll_\eps (abDMN)^\eps L^{-3}\frac{(D_1'D_2'D_2^*)^{1/2}(a'b'MN)^{3/2}}{(\ell_a'\ell_b')^{1/2}(aM+bN)},
\]
and hence
\begin{align}\label{boundE}
\mathcal{E}_{1;a,b,h}\ll_\eps (abhDMN)^\eps L^{-5/2}D\frac{(abMN)^{3/2}}{aM+bN}.
\end{align}

We now return to the principal term $\mathcal{M}_{a,h}$ in \eqref{afterVoronoi}. This corresponds to the product of the two constant terms after the applications of Lemma \ref{voronoi} and Lemma \ref{voronoi2}, and hence
\begin{align}\label{MVoronoi}
\mathcal{M}_{a,h}&=L(1,\psi)^2\sum_{\ell<2L}\ \sideset{}{^*}\sum_{k(\text{mod}\ \ell)}\rho_1(-kb',\ell_b')\rho_2(ka',\ell_a')e\Big(\frac{-hk}{\ell}\Big)\nonumber\\
&\qquad\qquad\int_{0}^{\infty}\int_{0}^{\infty}\Delta_\ell(ax- by-h)\phi(ax- by-h)f(x,y)dxdy.
\end{align}
We have
\begin{align*}
&\rho_1(-kb',\ell_b')\rho_2(ka',\ell_a')\\
&\qquad\quad=\frac{1}{\ell_a'\ell_b'}\big(\psi(\ell_b')+\mathbf{1}_{D|\ell_b'}\tau(\psi)\psi(-kb')\big)\bigg(\tau(\psi_1)\psi_1(ka')\psi_2\Big(\frac{\ell_a'}{D_1}\Big)+\tau(\psi_2)\psi_2(ka')\psi_1\Big(\frac{\ell_a'}{D_2}\Big)\bigg).
\end{align*}
So
\begin{align}\label{sumk1}
&\sideset{}{^*}\sum_{k(\text{mod}\ \ell)}\rho_1(-kb',\ell_b')\rho_2(ka',\ell_a')e\Big(\frac{-hk}{\ell}\Big)\nonumber\\
&\quad=\frac{1}{\ell_a'\ell_b'}\bigg(\tau(\psi_1)\psi(\ell_b')\psi_1(-a')\psi_2\Big(\frac{\ell_a'}{D_1}\Big)S_{\psi_1}(h,0;\ell)+\tau(\psi_2)\psi(\ell_b')\psi_2(-a')\psi_1\Big(\frac{\ell_a'}{D_2}\Big)S_{\psi_2}(h,0;\ell)\nonumber\\
&\qquad\qquad\qquad\qquad+\mathbf{1}_{D|\ell_b'}\tau(\psi)\tau(\psi_1)\psi(b')\psi_1(-a')\psi_2\Big(\frac{\ell_a'}{D_1}\Big)S_{\psi\psi_1}(h,0;\ell)\nonumber\\
&\qquad\qquad\qquad\qquad+\mathbf{1}_{D|\ell_b'}\tau(\psi)\tau(\psi_2)\psi(b')\psi_2(-a')\psi_1\Big(\frac{\ell_a'}{D_2}\Big)S_{\psi\psi_2}(h,0;\ell)\bigg).
\end{align}

For the double integrals in \eqref{MVoronoi}, a change of variables gives
\begin{align*}
&\int_{0}^{\infty}\int_{0}^{\infty}\Delta_\ell(ax-by-h)\phi(ax-by-h)f(x,y)dxdy\\
&\qquad\qquad=\frac{1}{ab}\int\int\Delta_\ell(u)\phi(u)f\Big(\frac{y+u+h}{a},\frac{y}{b}\Big) dydu.
\end{align*}
By \eqref{bdDelta} this is bounded by
\[
\ll_\eps (abMN)^\eps\frac{\min\{aM,bN\}}{ab}\ll_\eps (abMN)^\eps\frac{MN}{aM+bN}.
\]
Moreover, if $\ell\ll L^{1-\eps}$, then by [\textbf{\ref{DFI}}; (18)], this is equal to
\[
\frac{1}{ab}\int f\Big(\frac{y+h}{a},\frac{y}{b}\Big) dy+O_A\big(L^{-A}\big)
\]
for any fixed $A>0$. So in view of \eqref{sumk1} and the Weil bound, we can first restricted the sum over $\ell$ in \eqref{MVoronoi} to $\ell\ll L^{1-\eps}$, and then extend it to all $\ell\geq1$ at the cost of an error term of size
\begin{align*}
O_\eps\big((abhMN)^\eps L(1,\psi)^2 (h,D)^{1/2}L^{-1/2}MN(aM+bN)^{-1}\big).
\end{align*}
Hence
\begin{align}\label{Mab}
\mathcal{M}_{a,b,h}&=L(1,\psi)^2\mathfrak{S}_{a,b}(h)\frac{1}{ab}\int f\Big(\frac{y+h}{a},\frac{y}{b}\Big) dy\nonumber\\
&\qquad\qquad+O_\eps\big((abhMN)^\eps L(1,\psi)^2 (h,D)^{1/2}L^{-1/2}MN(aM+bN)^{-1}\big),
\end{align}
where $\mathfrak{S}_{a,b}(h)$ is defined in \eqref{mathfrakS1}.

Finally, combining \eqref{boundE} and \eqref{Mab}, and choosing $L=P^{-1/2}\min\{\sqrt{aM},\sqrt{bN}\}$ we obtain the proposition.
\end{proof}

\subsection{Evaluating $T_{M,N,h}^{-}(c)$}

Recall that
\[
G^-(m',n)=\int_0^\infty J_0 \Big(\frac{4\pi\sqrt{ m'x}}{cq\sqrt{D_{2}}} \Big)J_{1}\Big(\frac{4\pi\sqrt{agnx}}{cq}\Big)g(x,n) dx.
\]
Let $$P=1+\frac{\sqrt{agMN}}{cq}\qquad\text{and}\qquad M'=\frac{c^2q^2D_2P^2}{M}.$$ Note that $M'>agD_2N$. With respect to $x$, we do nothing if $m'\leq M'$, but if $m'>M'$ we integrate by parts in $x$ several times. Then we differentiate $i$ times in $m'$ and $j$ times in $n$. Using the bounds $J_0(x)\ll (1+x)^{-1/2}$ and $J_1(x)\ll x(1+x)^{-3/2}$ we get
\begin{equation}\label{boundforG-}
m'^in^j G^{-(ij)}(m',n)\ll_{A,i,j} \Big(1+\frac{m'}{M'}\Big)^{-A}\Big(1+\frac{n}{N}\Big)^{-A} \frac{M\sqrt{agMN}}{cq}P^{i+j-2}
\end{equation}
for any fixed $A,i,j\geq0$. Applying Proposition \ref{propoff} we obtain
\begin{align}\label{formulaTMNh-}
T_{M,N,h}^{-}(c)&=2\pi iL(1,\psi)^2\frac{\psi_1(-D_2)\psi_2(cq)}{\sqrt{D_{2}MN}}\mathfrak{S}_{1,agD_2}(h)W(h)+O_\eps\big(q^\eps D^{5/4}(ag)^{3/4}M^{1/2}(NP)^{1/4}\big),
\end{align}
where $\mathfrak{S}_{a,b}(h)$ is given in \eqref{mathfrakS1} and
\[
W(h)=\frac{1}{agD_2} \int\int J_0 \Big(\frac{4\pi\sqrt{ x(y+h)}}{cq\sqrt{D_{2}}} \Big)J_{1}\Big(\frac{4\pi\sqrt{xy}}{cq\sqrt{D_2}}\Big)g\Big(x,\frac{y}{agD_2}\Big) dxdy.
\] Note that as in the proof of Proposition \ref{propoff} we can truncate the sum over $\ell$ in $\mathfrak{S}_{1,agD_2}(h)$ to $\ell < cq$ and the resulting error is absorbed  in the above error term.

\section{The off-off-diagonal $\mathcal{M}_k^{OOD}$ - Initial manipulations}

In view of \eqref{formulaTMNh-} and \eqref{mathfrakS1} we write
\begin{equation}\label{formulaTMNh-1}
T_{M,N,h}^{-}(c)=\sum_{j=1}^{4}T_{j;M,N,h}^{-}(c)+O_\eps\big(q^\eps D^{5/4}(ag)^{3/4}M^{1/2}(NP)^{1/4}\big).
\end{equation}
Observe from \eqref{boundforG-} that the contribution of the terms with $m'\gg q^\eps M'$ is negligible. We also have $M'>agD_2N$ so we can restrict $|h|\leq H=q^\eps M'$ in \eqref{sumoverh} at the cost of a negligible error term. Hence, using Lemma \ref{Nelsonlemma} the contribution of the $O$-term in \eqref{formulaTMNh-1} to $\mathcal{M}_k^{OOD}$ is
\begin{align}\label{1stboundMkOOD}
&\ll_\eps q^\eps D^{9/4}\big(CM^{-1/2}N^{1/4}X^{5/2}+q^{-9/4}M^{5/8}N^{11/8}X^{19/4}\big)\nonumber\\
&\ll_\eps q^{-1/12+\eps}D^{5/3}X^{5/2}+q^{-1/4+\eps}D^{17/4}X^{19/4},
\end{align}
by \eqref{valueofC}.

All the four terms $T_{j;M,N,h}^{-}(c)$ are in similar forms, so we only illustrate how to estimate the contribution of $T_{1;M,N,h}^{-}(c)$ to $\mathcal{M}_k^{OOD}$.  We have
\begin{align*}\label{sumhT-}
&\sum_{h\ne0}S_{\psi_1}(h,0;cq)T_{1;M,N,h}^{-}(c)\\
&\qquad=2\pi iL(1,\psi)^2\frac{\tau(\psi_1)\psi_1(D_2)\psi_2(cq)}{\sqrt{D_{2}MN}}\sum_{\ell<cq}\frac{\psi(\ell')\psi_2(\ell/D_1)}{\ell\ell'}\sum_{h\ne0}S_{\psi_1}(h,0;cq)S_{\psi_1}(h,0;\ell)W(h),
\end{align*}
where $\ell'=\ell/(agD_2,\ell)$. Now we would like to re-extend the sum over $c$ to $c\geq 1$ but this step is not straightforward as in Section \ref{sectionMkOD}. Instead we denote the contribution of the complete sum $c\geq 1$ by $\mathcal{M}_{1;k}^{OOD}$ and the contribution of the terms with $c\geq C$ by $\mathcal{M}_{2;k}^{OOD}$,
\begin{align*}
\mathcal{M}_{1;k}^{OOD}&=\frac{2\pi iL(1,\psi)^2}{q^{2}}\sum_{\substack{d_1,d\\abg\leq X^2}}\frac{\psi(d_1)\psi(d)\mu(g)(1\star\psi)(b)\rho_2(abg)}{d_1d\sqrt{a}bg^{3/2}}\\
&\qquad \sum_{M,N}\sum_{c\geq1}\frac{\tau(\psi_1)\psi_1(D_2)\psi_2(cq)}{c^2\sqrt{D_{2}MN}}\sum_{\ell<cq}\frac{\psi(\ell')\psi_2(\ell/D_1)}{\ell\ell'}\sum_{h\ne0}S_{\psi_1}(h,0;cq)S_{\psi_1}(h,0;\ell)W(h)
\end{align*}
and
\begin{align*}
\mathcal{M}_{2;k}^{OOD}&=\frac{2\pi iL(1,\psi)^2}{q^{2}}\sum_{\substack{d_1,d\\abg\leq X^2}}\frac{\psi(d_1)\psi(d)\mu(g)(1\star\psi)(b)\rho_2(abg)}{d_1d\sqrt{a}bg^{3/2}}\\
&\qquad \sum_{M,N}\sum_{c\geq C}\frac{\tau(\psi_1)\psi_1(D_2)\psi_2(cq)}{c^2\sqrt{D_{2}MN}}\sum_{\ell<cq}\frac{\psi(\ell')\psi_2(\ell/D_1)}{\ell\ell'}\sum_{h\ne 0}S_{\psi_1}(h,0;cq)S_{\psi_1}(h,0;\ell)W(h).\nonumber
\end{align*} 
Then
\begin{equation}\label{501}
\mathcal{M}_{k}^{OOD}=\mathcal{M}_{1;k}^{OOD}-\mathcal{M}_{2;k}^{OOD}+O_\eps\big( q^{-1/12+\eps}D^{5/3}X^{5/2}\big)+O_\eps\big(q^{-1/4+\eps}D^{17/4}X^{19/4}\big),
\end{equation}
by \eqref{1stboundMkOOD}. 


We will estimate $\mathcal{M}_{1;k}^{OOD}$ and $\mathcal{M}_{2;k}^{OOD}$ in the next two sections. The term $\mathcal{M}_{1;k}^{OOD}$ is delicate and we first provide an outline. Out strategy is as follows:
\begin{enumerate}
\item We start by bringing in the Mellin transforms of the Bessel $J$-functions and formula \eqref{formulaV} to express $W(h)$ as a triple contour integral.
\item We then convert the Gauss-Ramanujan sums to divisibility conditions.
\item The sum involving $h$ after Step 1 is
\[
\sum_{h>0}\int_{(w_0)}\int_{(v_0)}\int_{(u_0)}S_{\psi_1}(h,0;cq)S_{\psi_1}(h,0;\ell)h^{1-u-v-w}dudvdw,
\]
provided that
\[
\begin{cases}
u_0,v_0,w_0>0,\\
u_0+w_0<1,\\
v_0+w_0<1,\\
 u_0+v_0+w_0>1.
\end{cases}
\]
We want to move the sum over $h$ inside the integrals and replace it with the zeta-functions. This requires $\text{Re}(u+v+w)>2$, but note that the above conditions force $u_0+v_0+w_0<2$. We therefore have to move the contours of integration several times to make sure the $h$-sum converges absolutely. The extra factor $G(u)$ in the integral expression of $V_k$ in  \eqref{formulaV} is crucial in this step.
\end{enumerate}

\section{Evaluating $\mathcal{M}_{1;k}^{OOD}$}\label{afeood}

\subsection{Bringing in the Mellin transforms}

Removing the partition of unity we obtain
\begin{equation*}
\sum_{M,N}W(h)=\frac{1}{\sqrt{agD_2}} \int\int J_0 \Big(\frac{4\pi\sqrt{ x(y+h)}}{cq\sqrt{D_{2}}} \Big)J_{1}\Big(\frac{4\pi\sqrt{xy}}{cq\sqrt{D_{2}}}\Big)V_{k}\Big(\frac{d_1^2x}{Q}\Big)V_{k}\Big(\frac{d^2bgy}{aD_2Q}\Big) \frac{dxdy}{\sqrt{xy}}.
\end{equation*}
We bring in the Mellin transforms of $J_0(x)$ and $J_1(x)$: 
\[
J_0(x)=\frac{1}{2\pi i}\int_{(v_0)}\frac{\Gamma(v)}{\Gamma(1-v)}\Big(\frac{x}{2}\Big)^{-2v}dv
\]
and
\[
J_1(x)=\frac{1}{2\pi i}\int_{(w_0)}\frac{\Gamma(w)}{\Gamma(2-w)}\Big(\frac{x}{2}\Big)^{1-2w}dw
\]
for $v_0,w_0>0$ (see, for instance, [\textbf{\ref{GR}}; 6.561(14), 8.412(4)]). These together with the integral formula \eqref{formulaV} for the second factor $V_k$ imply that
\begin{align*}
&\sum_{M,N}W(h)=\frac{1}{ \sqrt{agD_2}}\frac{1}{(2\pi i)^3}\int_{(w_0)}\int_{(v_0)}\int_{(u_0)}G(u)\frac{\Gamma(1+u)^2\Gamma(v)\Gamma(w)}{\Gamma(1-v)\Gamma(2-w)}\Big(\frac{d^2bg}{aD_2Q}\Big)^{-u}\\
&\qquad\qquad\Big(\frac{2\pi}{cq\sqrt{D_{2}}}\Big)^{1-2v-2w}\bigg(\int x^{-(v+w)}V_{k}\Big(\frac{d_1^2x}{Q}\Big)dx\bigg)\bigg(\int (y+h)^{-v}y^{-(u+w)}dy\bigg)\frac{dudvdw}{u^{k+1}}.
\end{align*}

By \eqref{formulaV} and the inversion Mellin transform, the $x$-integral is 
\[
\Big(\frac{Q}{d_1^2}\Big)^{1-v-w}G(1-v-w)\frac{\Gamma(2-v-w)^2}{(1-v-w)^{k+1}}
\]
if
\begin{equation}\label{contour1}
v_0+w_0<1.
\end{equation}
For the $y$-integral, note that if $h<0$ then the integral is restricted to $y>-h$, and if $h>0$ then it is over $y>0$. For absolute convergence we need to impose the conditions
\begin{equation}\label{contour2}\begin{cases}
 u_0+v_0+w_0>1,\\
u_0+w_0<1.
\end{cases}\end{equation} Under these assumptions, the $y$-integral is equal to (see, for instance, [\textbf{\ref{GR}}; 17.43(21), 17.43(22)])
\[
|h|^{1-u-v-w}\times\begin{cases}
\frac{\Gamma(u+v+w-1)\Gamma(1-v)}{\Gamma(u+w)}& \qquad\text{if }h<0,\\
\frac{\Gamma(u+v+w-1)\Gamma(1-u-w)}{\Gamma(v)} & \qquad\text{if }h>0 .
\end{cases}
\]
Hence provided that $u_0,v_0,w_0$ satisfy conditions \eqref{contour1} and \eqref{contour2}  we have
\begin{align}\label{MOODintegral}
\mathcal{M}_{1;k}^{OOD}&=L(1,\psi)^2\sum_{h>0}\frac{1}{(2\pi i)^2}\int_{(w_0)}\int_{(v_0)}\int_{(u_0)}\gamma(u,v,w)I_h(u,v,w)\frac{dudvdw}{u^{k+1}(1-v/2-w)^{k+1}},
\end{align}
where
\begin{align}\label{gammauvw}
\gamma(u,v,w)&=G(u)G(1-v-w)\frac{\Gamma(1+u)^2\Gamma(w)\Gamma(2-v-w)^2\Gamma(u+v+w-1)}{\Gamma(2-w)}\nonumber\\
&\qquad\qquad\Big(\frac{\Gamma(v)}{\Gamma(u+w)}+\frac{\Gamma(1-u-w)}{\Gamma(1-v)}\Big)
\end{align}
and
\begin{align*}
&I_h(u,v,w)=L(1+2u,\psi)L(3-2v-2w,\psi)\frac{(2\pi)^{1-2v-2w}}{q^{3-2v-2w}}Q^{1+u-v-w}\sum_{\substack{abg\leq X^2}}\frac{\mu(g)(1\star\psi)(b)\rho_2(abg)}{a^{1-u}b^{1+u}g^{2+u}}\\
&\qquad\qquad\sum_{c\geq1}\frac{\tau(\psi_1)\psi_1(D_2)\psi_2(cq)}{c^{3-2v-2w}D_2^{3/2-u-v-w}}\sum_{\ell\geq1}\frac{\psi(\ell')\psi_2(\ell/D_1)}{\ell\ell'}S_{\psi_1}(h,0;cq)S_{\psi_1}(h,0;\ell)h^{1-u-v-w}.
\end{align*}

To make sure the sums over $a$ and $b$ are controllable we want $u_0\asymp 1/\log q$. Later we also want $(1-w_0)\asymp 1/\log q$. So we choose
\[
u_0=\frac{2}{\log q},\qquad v_0=\frac{2}{\log q}\qquad\text{and}\qquad w_0=1-\frac{3}{\log q}.
\]

\subsection{Converting the Gauss-Ramanujan sums}

We write $c=c_1c_2$, $\ell=\ell_1\ell_2$ and $h=h_1h_2$, where $c_1\ell_1h_1|D_1^\infty$ and $(c_2\ell_2h_2,D_1)=1$. Using Lemma \ref{Nelsonlemma} the product of the Gauss-Ramanujan sums vanishes unless $D_1h_1 = c_1=\ell_1$, and in that case we have
\begin{equation}\label{GRsum1}
S_{\psi_1}(h,0;cq)=h_1\tau(\psi_1)\psi_1(c_2qh_2)\sum_{\substack{rc_2'=c_2q\\r|h_2}}\mu(c_2')r
\end{equation}
and
\begin{equation}\label{GRsum2}
S_{\psi_1}(h,0;\ell)=h_1\tau(\psi_1)\psi_1(\ell_2h_2)\sum_{\substack{s\ell_2'=\ell_2\\s|h_2}}\mu(\ell_2')s.
\end{equation}

There are two cases: $(q,c_2')=1$, which will force $q|r|h_2$, or $q|c_2'$. The former leads to two different cases to consider: $(q,s)=1$ or $q|s|h$. 

\subsection{The $h$-sum}

We consider the terms with $(q,c_2's)=1$. The other cases can be dealt with in the same way and can be adsorbed in the expression \eqref{Itozetafc} below.  By \eqref{GRsum1} and \eqref{GRsum2} we get
\begin{align*}
&\sum_{c\geq1}\frac{\tau(\psi_1)\psi_1(D_2)\psi_2(cq)}{c^{3-2v-2w}D_2^{3/2-u-v-w}}\sum_{\ell< q}\frac{\psi(\ell')\psi_2(\ell/D_1)}{\ell\ell'}\sum_{h>0}S_{\psi_1}(h,0;cq)S_{\psi_1}(h,0;\ell)h^{1-u-v-w}\\
&\quad=\frac{\psi(q)q^{2-u-v-w}}{L(3-2v-2w,\psi)D^{3/2-u-v-w}}\sum_{\substack{D_1D_2=D\\D_1|ag}}\frac{\tau(\psi_1)^3\psi_1(D_2)\psi_2(D_1)}{D_1^{5/2+u-v-w}}\sum_{h_1|D_1^\infty}\frac{1}{h_1^{2+u-v-w}}\\
&\qquad\sum_{(h_2,D_1)=1}h_2^{1-u-v-w} \sum_{\substack{r|h_2\\s|h_2\\(s,q)=1}}\frac{\psi(rs)}{r^{2(1-v-w)}s}\sum_{\ell\geq1}\frac{\mu(l)\psi(h_1s\ell/(ag/D_1,h_1s\ell))\psi(\ell)(ag/D_1,h_1s\ell)}{\ell^2}.
\end{align*}
Writing
\[
\sum_{(h_2,D_1)=1} \sum_{\substack{r|h_2\\s|h_2\\(s,q)=1}}=\sum_{\substack{r,s\geq1\\(s,q)=1}}\sum_{\substack{[r,s]|h_2\\(h_2,D_1)=1}}
\]
and using the Mobius inversion formula we see that
\begin{align}\label{Itozetafc}
\sum_{h>0}I_h(u,v,w)&=D^{-1/2+2u}\frac{L(1+2u,\psi)L(1+u-v-w,\psi)L_q(u+v+w,\psi)}{\zeta_{qD}(1+2u)}\nonumber\\
&\qquad\qquad \zeta_{qD}(2+u-v-w)\zeta(u+v+w-1)J(u,v,w),
\end{align}
where
\begin{align}\label{propertiesJ}
J(u,v,w)\ll_\eps (\log q)^{16+\eps}X^{2\text{Re}(u)}
\end{align}
uniformly for $\text{Re}(u)>0$, $\text{Re}(v)=v_0$ and $\text{Re}(w)=w_0$.

For convergence we need
\begin{align}\label{contour3}
\begin{cases}
\text{Re}(u)>\text{Re}(v+w),\\
\text{Re}(u+v+w)>2.
\end{cases}
\end{align}
Hence in order to be able to use \eqref{Itozetafc} we need to move the contours of integration several times.

\subsection{Moving the contours}\label{movingcontour}

We first move the $u$-contour in \eqref{MOODintegral} to the right to $\text{Re}(u)=1+2/\log q$, encountering a simple pole at $u=1-w$ from the gamma factor $\Gamma(1-u-w)$ in $\gamma(u,v,w)$. In doing so we obtain 
\[
\mathcal{M}_{1;k}^{OOD}=\mathcal{M}_{1;k}^{OOD'}-\mathcal{R}_1^{OOD},
\]
where $\mathcal{M}_{1;k}^{OOD'}$ is the new integral and
\begin{align}\label{R1OOD}
\mathcal{R}_1^{OOD}&=-L(1,\psi)^2\sum_{h>0}\frac{1}{(2\pi i)^2}\int_{(w_0)}\int_{(v_0)}G(1-w)G(1-v-w)\\
&\qquad\qquad \frac{\Gamma(2-w)\Gamma(w)\Gamma(2-v-w)^2\Gamma(v)}{\Gamma(1-v)}I_h(1-w,v,w)\frac{dvdw}{(1-w)^{k+1}(1-v-w)^{k+1}}.\nonumber
\end{align}

With $\mathcal{M}_{1;k}^{OOD'}$ we can move the sum over $h$ inside the integrals and replace the sum over $h$ using \eqref{Itozetafc} (see condition \eqref{contour3}). We then shift the $u$-contour to $\text{Re}(u)=4/\log q$, crossing no poles. Note that the pole of $\zeta(u+v/2+w-1)$ is cancelled as $\gamma(2-v/2-w,v,w)=0$. Hence
\begin{align*}
\mathcal{M}_{1;k}^{OOD'}&=L(1,\psi)^2\frac{1}{(2\pi i)^3}\int_{(w_0)}\int_{(v_0)}\int_{(4/\log q)}\gamma(u,v,w)D^{-1/2+2u}\\
&\qquad\quad\frac{L(1+2u,\psi)L(1+u-v-w,\psi)L_q(u+v+w,\psi)}{\zeta_{qD}(1+2u)}\\
&\qquad\qquad\quad\zeta_{qD}(2+u-v-w)\zeta(u+v+w-1)J(u,v,w)\frac{dudvdw}{u^{k+1}(1-v-w)^{k+1}}.
\end{align*}
Bounding this trivially using \eqref{propertiesJ} we get
\[
\mathcal{M}_{1;k}^{OOD'}\ll_\eps L(1,\psi)^2(\log q)^{2k+20+\eps}.
\]

For $\mathcal{R}_1^{OOD}$ we move the $w$-contour in \eqref{R1OOD} to $\text{Re}(w)=-2/\log q$, followed by moving the $v$-contour to $\text{Re}(v)=1+1/\log q$. Both times we cross no poles as the pole at $w=0$ of $\Gamma(w)$ is cancelled by the factor $G(1-w)$. Along these new contours, conditions \eqref{contour3} are now satisfied and hence the sum over $h$ can be moved inside the integrals and replaced by \eqref{Itozetafc}. So
\begin{align*}
&\mathcal{R}_1^{OOD}=-L(1,\psi)^2\frac{1}{(2\pi i)^2}\int_{(-2/\log q)}\int_{(1+1/\log q)}G(1-w)G(1-v-w)\\
&\qquad\quad \frac{\Gamma(2-w)\Gamma(w)\Gamma(2-v-w)^2\Gamma(v)}{\Gamma(1-v)}D^{3/2-2w} \frac{L(3-2w,\psi)L(2-v-2w,\psi)L_q(1+v,\psi)}{\zeta_{qD}(3-2w)}\\
&\qquad\qquad\qquad\zeta_{qD}(3-v-2w)\zeta(v)J(1-w,v,w)\frac{dvdw}{(1-w)^{k+1}(1-v/2-w)^{k+1}}.\nonumber
\end{align*}

We next move the $v$-contour back to $\text{Re}(v)=v_0$, and then the $w$-contour back to $\text{Re}(w)=w_0$. Again there are no poles as $G(1-w)\Gamma(w)$ has no pole at $w=0$ and $\zeta(v)/\Gamma(1-v)$ has no pole at $v=1$. We bound the resulting integral trivially using  \eqref{propertiesJ} and get
\[
\mathcal{R}_1^{OOD}\ll_\eps L(1,\psi)^2(\log q)^{2k+19+\eps}.
\]
It follows that
\begin{equation}\label{502}
\mathcal{M}_{1;k}^{OOD}\ll_\eps L(1,\psi)^2(\log q)^{2k+20+\eps}.
\end{equation}

\section{Evaluating $\mathcal{M}_{2;k}^{OOD}$}\label{IkODlast}

To estimate $\mathcal{M}_{2;k}^{OOD}$ 
%
we shall use the following lemma.

\begin{lemma}\label{orthoRam}
Let $c_1,c_2\in\mathbb{N}$ with $c_1\ne c_2$ and let $f$ be a smooth compactly supported function on $\mathbb{R}$. Then we have
\[
\sum_h S_{\chi}(h,0;c_1)S_{\chi}(h,0;c_2)f(h)=\sum_{u\ne 0}\widehat{f}\Big(\frac{u}{[c_1,c_2]}\Big)\nu_{c_1,c_2}(u),
\]
with $|\nu_{c_1,c_2}(u)|\leq \varphi((c_1,c_2))$.
\end{lemma}
\begin{proof}
We split the sum into progressions $h\equiv a(\text{mod}\ [c_1,c_2])$ and then apply the Poisson summation formula. In doing so we obtain
\begin{align*}
&\frac{1}{[c_1,c_2]}\sum_{a(\text{mod}\ [c_1,c_2])} S_{\chi}(a,0;c_1)S_{\chi}(a,0;c_2) \sum_{u}e\Big(\frac{-au}{[c_1,c_2]}\Big)\widehat{f}\Big(\frac{u}{[c_1,c_2]}\Big)\\
&\qquad\quad=\sum_{u}\widehat{f}\Big(\frac{u}{[c_1,c_2]}\Big)\nu_{c_1,c_2}(u),
\end{align*}
where 
\begin{align*}
\nu_{c_1,c_2}(u)&=\frac{1}{[c_1,c_2]}\ \sideset{}{^*}\sum_{v(\text{mod}\ c_1)}\ \sideset{}{^*}\sum_{w(\text{mod}\ c_2)}\ \chi(vw)\sum_{a(\text{mod}\ [c_1,c_2])}e\bigg(a\Big(\frac{v}{c_1}+\frac{w}{c_2}-\frac{u}{[c_1,c_2]}\Big)\bigg)\\
&=\sideset{}{^*}\sum_{v(\text{mod}\ c_1)}\ \sideset{}{^*}\sum_{w(\text{mod}\ c_2)}\mathbf{1}_{v[c_1,c_2]/c_1+w[c_1,c_2]/c_2\equiv u(\text{mod}\ [c_1,c_2])}\ \chi(vw).
\end{align*}
Observe that the congruence equation has no solutions unless $(u,[c_1c_2]/(c_1,c_2))=1$ (so, in particular, $u\ne0$ as $c_1\ne c_2$), and in that case the number
of solutions is less than or equal to $\varphi((c_1,c_2))$. This completes the proof of the lemma.
\end{proof}

In order to apply this we first write
\[
W(h)=\frac{1}{agD_2} \int_{M}^{\infty}\int_{agD_2MN/x}^{4agD_2MN/x} J_0 \Big(\frac{4\pi\sqrt{ x(y+1)}}{cq\sqrt{D_{2}}} \Big)J_{1}\Big(\frac{4\pi\sqrt{xy}}{cq\sqrt{D_2}}\Big)f(h) dydx,
\]
where
\[
f(h)=g\Big(\frac{x}{h},\frac{yh}{agD_2}\Big),
\]
if $h>0$, and a a similar representation holds for $h<0$. Integrating by parts twice we see that the Fourier transform of $f$ is bounded by
\[
\widehat{f}(u)\ll \big(1+|u|xM^{-1}\big)^{-2}xM^{-1}.
\]
Hence by Lemma \ref{orthoRam} we obtain
\[
\sum_{h> 0}S_{\psi_1}(h,0;cq)S_{\psi_1}(h,0;\ell)f(h)\ll cq\ell.
\]
Moreover we have
\[
\int_{M}^{\infty}\int_{agD_2MN/x}^{4agD_2MN/x} \bigg|J_0 \Big(\frac{4\pi\sqrt{ x(y+1)}}{cq\sqrt{D_{2}}} \Big)J_{1}\Big(\frac{4\pi\sqrt{xy}}{cq\sqrt{D_2}}\Big)\bigg| dydx\ll \frac{(agD_2MN)^{3/2}}{cq\sqrt{D_2}}(\log q),
\]
by using the bounds $J_0(x)\ll (1+x)^{-1/2}$ and $J_1(x)\ll x$. Thus
\[
\sum_{h> 0}S_{\psi_1}(h,0;cq)S_{\psi_1}(h,0;\ell)W(h)\ll (ag)^{1/2}(MN)^{3/2}\ell(\log q).
\]

The same bound holds for the sum over $h<0$. Putting this into  $\mathcal{M}_{2;k}^{OOD}$ we obtain
\begin{align*}
\mathcal{M}_{2;k}^{OOD}&\ll_\eps \frac{L(1,\psi)^2}{q^{2-\eps}} \frac{MN}{C\sqrt{D}}X^2\ll_\eps q^{-1/6+\eps}D^{11/6}X^2,
\end{align*}
by \eqref{valueofC}. This together with \eqref{firstestimate}, \eqref{500}, \eqref{501} and \eqref{502} complete the proof of Proposition \ref{propE_k}.

\section{Proofs of the main theorems}\label{sec:proof of main thm}

As mentioned in the introduction, Theorem \ref{Maintheorem} is a straightforward consequence of Theorem \ref{Maintheorem2}. Here we give the easy proof.
\begin{proof}[Proof of Theorem \ref{Maintheorem} assuming Theorem \ref{Maintheorem2}]
We may assume $L(1,\psi) \leq (\log D)^{-50}$. By the work of Gross and Zagier [\textbf{\ref{GZ}}], the product
\[
\prod_{\substack{f\in S_2^*(q)\\r_f=1}}L(f,s)
\]
is the $L$-function of a
quotient of $J_0(q)$ with rank exactly equal to its dimension. Hence the lower bound
\begin{align*}
\textup{rank}(J_0(q)) \geq \left(\frac{1}{2} + O\Big(\frac{\log\log\log q}{\log\log q} \Big) \right) \textup{dim}(J_0(q))
\end{align*}
follows from Theorem \ref{Maintheorem2}, since the proportion of odd forms is $1/2 + O(q^{-\varepsilon})$.

For the upper bound, it follows from Theorem \ref{Maintheorem2} that it suffices to prove
\begin{align*}
\frac{1}{|S_2^*(q)|}\sum_{\substack{f \in S_2^*(q) \\ r_f \geq 3}} r_f \ll \sqrt{\frac{\log\log\log q}{\log\log q}} .
\end{align*}
By Cauchy-Schwarz's inequality the left-hand side is
\begin{align*}
&\leq \frac{1}{|S_2^*(q)|}\bigg(\sum_{\substack{f \in S_2^*(q) \\ r_f \geq 3}} 1 \bigg)^{1/2}\bigg(\sum_{\substack{f \in S_2^*(q)}} r_f^2 \bigg)^{1/2}\ll q^{-1/2}\bigg(\sum_{\substack{f \in S_2^*(q) \\ r_f \geq 3}} 1 \bigg)^{1/2},
\end{align*}
by Theorem 8.1 in [\textbf{\ref{KMV1}}].
By Theorem \ref{Maintheorem2} again, this is
\begin{align*}
\ll \sqrt{\frac{\log\log\log q}{\log\log q}},
\end{align*}
as required.
\end{proof}

As stated above, Theorem \ref{Maintheorem} is a result about the natural average of analytic ranks of $L$-functions. We first prove a result for the harmonic average, and then describe the necessary modifications to obtain Theorem \ref{Maintheorem}.

\begin{theorem}\label{thm:harmonic average thm}
Let $C \geq 750$ be a fixed real number. Let $D$ be large, and let $\psi$ be a real, odd, primitive Dirichlet character modulo $D$.Then for any $\varepsilon > 0$ and any prime $q$ satisfying
\begin{align*}
D^{750} \leq q \leq D^{C}
\end{align*}
we have 
\begin{align*}
\sideset{}{^h}\sum_{\substack{f \in S_2^*(q) \\ r_f \leq 1}}1 = 1 + O_\varepsilon( L(1,\psi)(\log q)^{45+\varepsilon}) + O_\varepsilon( L(1,\psi)^2(\log q)^{56+\varepsilon}) + O_A((\log q)^{-A})
\end{align*}
if $\psi(q) = 1$, and
\begin{align*}
\sideset{}{^h}\sum_{\substack{f \in S_2^*(q) \\ r_f \leq 2}}1 &= 1+ O \left(\frac{\log \log (1/L(1,\psi)\log D)}{\log (1/L(1,\psi)\log D)} \right) + O_\varepsilon( L(1,\psi)(\log q)^{45+\varepsilon})\\
&\qquad\qquad + O_\varepsilon( L(1,\psi)^2(\log q)^{56+\varepsilon}) 
\end{align*}
if $\psi(q) = -1$ and $L(1,\psi)(\log D)=o(1)$.
\end{theorem}
\begin{proof}
We prove the second part of the theorem, since the first is similar and easier. Define
\begin{align*}
S_1 &= \sideset{}{^h}\sum_{f\in S_2^*(q)}\Lambda_{f,\psi}''\Big(\frac{1}{2}\Big)M_{f,\psi}
\end{align*}
and
\begin{align*}
S_2 &= \sideset{}{^h}\sum_{f\in S_2^*(q)}\Lambda_{f,\psi}''\Big(\frac{1}{2}\Big)^2M_{f,\psi}^2.
\end{align*}
By Cauchy-Schwarz's inequality, we have
\begin{align*}
\sideset{}{^h}\sum_{\substack{f\in S_2^*(q) \\ \Lambda_{f,\psi}''(1/2) \neq 0}}1 \geq \frac{S_1^2}{S_2}.
\end{align*}
Letting $X=D^{24}$, together Propositions \ref{mainprop1} and \ref{mainprop3} imply that
\begin{align*}
\sideset{}{^h}\sum_{\substack{f\in S_2^*(q) \\ \Lambda_{f,\psi}''(1/2) \neq 0}}1  &=  1 + O \left(\frac{\log \log (1/L(1,\psi)\log D)}{\log (1/L(1,\psi)\log D)} \right)+ O_\varepsilon( L(1,\psi)(\log q)^{45+\varepsilon})\\
&\qquad\qquad+O_\varepsilon( L(1,\psi)^2(\log q)^{56+\varepsilon}).
\end{align*}

By the product rule we see that
\begin{align*}
\Lambda_{f,\psi}''\Big(\frac{1}{2}\Big) &= \Lambda''\Big(f,\frac{1}{2}\Big)\Lambda\Big(f \otimes \psi,\frac{1}{2}\Big)+2\Lambda'\Big(f,\frac{1}{2}\Big)\Lambda'\Big(f \otimes \psi,\frac{1}{2}\Big)+\Lambda\Big(f,\frac{1}{2}\Big)\Lambda''\Big(f \otimes \psi,\frac{1}{2}\Big).
\end{align*}
If $f$ is odd then by root number considerations we see that
\begin{align*}
\Lambda_{f,\psi}''\Big(\frac{1}{2}\Big) &= 2\Lambda'\Big(f,\frac{1}{2}\Big)\Lambda'\Big(f \otimes \psi,\frac{1}{2}\Big).
\end{align*}
If the left-hand side is nonzero then we must have $\Lambda'(f,1/2) \neq 0$. Again by root number considerations we see that $\Lambda'(f,1/2)' = L'(f,1/2)$, so $r_f=1$.

If $f$ is even, then
\begin{align*}
\Lambda_{f,\psi}''\Big(\frac{1}{2}\Big) &=\Lambda''\Big(f,\frac{1}{2}\Big)\Lambda\Big(f \otimes \psi,\frac{1}{2}\Big)+\Lambda\Big(f,\frac{1}{2}\Big)\Lambda''\Big(f \otimes \psi,\frac{1}{2}\Big),
\end{align*}
and therefore at least one of $\Lambda(f,1/2)$ or $\Lambda''(f,1/2)$ is nonzero. If $\Lambda(f,1/2) \neq 0$ then $r_f=0$, so suppose that $\Lambda(f,1/2) = 0$ and $\Lambda''(f,1/2) \neq 0$. Taking derivatives and noting that $L(f,1/2)=L'(f,1/2)=0$, we deduce that $\Lambda''(f,1/2) = L''(f,1/2)$, and therefore $r_f = 2$.
\end{proof}

To deduce Theorem \ref{Maintheorem2} we just need to remove the harmonic weights $w_f$ in Propositions \ref{mainprop1}--\ref{mainprop3} so that the main propositions also hold for the natural average,
\begin{equation*}
\sideset{}{^n}\sum_{f\in S_2^*(q)} A_f:=\frac{1}{|S_{2}^{*}(q)|}\sum_{f\in S_{2}^{*}(q)}A_f.
\end{equation*}
This technique has been done several times (see, for instance, [\textbf{\ref{KM}}, \textbf{\ref{KMV}}, \textbf{\ref{KMV1}}]), so here we shall only illustrate the method for the mollified first moment of $\Lambda_{f,\psi}'(1/2)$,
\begin{equation*}
\sideset{}{^n}\sum_{f\in S_2^*(q)}\Lambda_{f,\psi}'\Big(\frac{1}{2}\Big)M_{f,\psi}.
\end{equation*}
In doing this we need the following general lemma from [\textbf{\ref{KM}}].

\begin{lemma}
Let $(A_f)_{f\in S_{2}^{*}(q)}$ be a family of complex numbers satisfying
\begin{equation}\label{crit1}
 \sideset{}{^h}\sum_{f\in S_2^*(q)} |A_f|\ll (\log q)^A
 \end{equation} and
\begin{equation}\label{crit2}
\max_{f\in S_{2}^{*}(q)}w_f|A_f|\ll q^{-\delta}
\end{equation}
for some absolute $A,\delta>0$.
Then for all $\kappa>0$, there exists $\eps=\eps(A,\delta)>0$ such that
\begin{eqnarray*}
\sideset{}{^n}\sum_{f\in S_2^*(q)} A_f=\frac{1}{\zeta(2)}\,\sideset{}{^h}\sum_{f\in S_2^*(q)} w_f(q^\kappa)A_f+O\big(q^{-\eps}\big),
\end{eqnarray*}
where
\begin{equation*}
w_f(q^\kappa)=\sum_{\ell m^2\leq q^\kappa}\frac{\lambda_f(\ell^2)}{\ell m^2}.
\end{equation*}
\end{lemma}

We shall apply this lemma for $A_f=\Lambda_{f,\psi}'(1/2)M_{f,\psi}$. Condition \eqref{crit1} follows immediately from Proposition \ref{mainprop2} and Cauchy-Schwarz's inequality. For condition \eqref{crit2}, it is known that $w_f\ll \log q/q$ [\textbf{\ref{GHL}}]. Hence \eqref{crit2} is satisfied provided that $DX\ll q^{1-\delta}$ for some $\delta>0$, using the convexity bound $\Lambda_{f,\psi}'(1/2)\ll_\eps (qD)^{1/2+\varepsilon}$ and the trivial bound $M_{f,\psi}\ll_\eps X^{1/2+\varepsilon}$. Thus we are left with estimating the sum
\begin{eqnarray*}
I=\frac{1}{\zeta(2)}\,\sideset{}{^h}\sum_{f\in S_2^*(q)} w_f(q^\kappa)\Lambda_{f,\psi}'\Big(\frac{1}{2}\Big)M_{f,\psi}.
\end{eqnarray*}

As before, we can remove the condition $(d,q)=1$ in the expression for $\Lambda_{f,\psi}'(1/2)$ in Lemma \ref{afe} due to the decay of the function $V_1$. Using \eqref{mollifier} and applying Lemma \ref{ortho} we get
\[
\Lambda_{f,\psi}'\Big(\frac{1}{2}\Big)M_{f,\psi}=\big(1+\psi(q)\big)\sum_{\substack{d\\ab\leq X}}\frac{\psi(d)\rho_1(ab)}{d\sqrt{a}b}\sum_{n}\frac{(1\star\psi)(bn)\lambda_f(an)}{\sqrt{n}}V_{1}\Big(\frac{d^2bn}{Q}\Big).
\]
Hence
\[
I=\frac{1+\psi(q)}{\zeta(2)}\sum_{\ell m^2\leq q^\kappa}\sum_{\substack{d\\ab\leq X}}\frac{\psi(d)\rho_1(ab)}{d\sqrt{a}b}\sum_{n}\frac{(1\star\psi)(bn)}{\ell m^2\sqrt{n}}V_{1}\Big(\frac{d^2bn}{Q}\Big)\,\sideset{}{^h}\sum_{f\in S_2^*(q)}\lambda_f(\ell^2)\lambda_f(an).
\]
From Lemma \ref{Petersson} we obtain
\begin{align*}
I=\frac{1+\psi(q)}{\zeta(2)}\sum_{\ell m^2\leq q^\kappa}\sum_{\substack{d\\ab\leq X\\ an=\ell^2}}\frac{\psi(d)\rho_1(ab)(1\star\psi)(bn)}{dabm^2n}V_{1}\Big(\frac{d^2bn}{Q}\Big)+O_\varepsilon\big(q^{-1/2+\kappa+\varepsilon}DX\big).
\end{align*}
By \eqref{formulaV} we can write the main term as
\begin{align*}
&\frac{1+\psi(q)}{\zeta(2)}\frac{1}{2\pi i}\int_{(1)}\Gamma(1+u)^2Q^uL(1+2u,\psi)\mathcal{S}(u)\frac{du}{u^{2}},
\end{align*}
where
\begin{align*}
\mathcal{S}(u)=\sum_{\ell m^2\leq q^\kappa}\sum_{\substack{ab\leq X\\ an=\ell^2}}\frac{\rho_1(ab)(1\star\psi)(bn)}{ab^{1+u}m^2n^{1+u}}.
\end{align*}

It is clear that $\mathcal{S}(u)\ll_\eps X^{-\text{Re}(u)+\eps}$ and trivially we have $\mathcal{S}'(0)\ll_\eps(\log q)^{5+\eps}$. We move the contour to $\text{Re}(u)=-1/2+\varepsilon$, crossing a double pole at $u=0$, and the new integral is 
\[
\ll_\eps q^{\eps}\Big(\frac{Q}{X}\Big)^{-1/2}D^{1/2}\ll_\eps q^{-1/2+\eps}X^{1/2}.
\]
 The contribution of the residue is
\[
\frac{2\big(1+\psi(q)\big)}{\zeta(2)}\mathcal{S}(0)L'(1,\psi)+O_\eps\big(L(1,\psi)(\log q)^{5+\eps}\big).
\]
Thus it remains to prove the following lemma.

\begin{lemma}
We have
\begin{align*}
\mathcal{S}(0)&=\zeta(2)\mathfrak{S}_1+O_\eps\big(L(1,\psi)(\log q)^{11+\eps}\big)+O_A\big((\log q)^{-A}\big),
\end{align*}
where $\mathfrak{S}_1$ is given by \eqref{singularseries1}.
\end{lemma}
\begin{proof}
The sum on $m$ in $\mathcal{S}(0)$ should obviously yield the factor of $\zeta(2)$ out in front. To execute this sum rigorously, we first separate $\ell$ and $m$ from each other. By trivial estimation, the contribution from $m > q^\varepsilon$ is
\begin{align*}
&\ll_\eps (\log q)^\varepsilon \sum_{m > q^\varepsilon }\frac{1}{m^2} \sum_{\substack{a,b,n \leq q}}\frac{\tau(ab)\tau(bn)}{abn} \ll_\eps q^{-\varepsilon} \bigg( \sum_{a \leq q} \frac{\tau(a)}{a}\bigg)^2\bigg( \sum_{b \leq q} \frac{\tau(b)^2}{b}\bigg)\ll_\eps q^{-\varepsilon}.
\end{align*}

With the size of $m$ reduced, we now wish to extend the sum on $\ell$ to infinity. Of course, the sum on $\ell$ is not really infinite because we have the condition $an = \ell^2$, but it is useful to think in these terms. The contribution from $\ell > q^\kappa/m^2 \geq q^{\kappa-2\varepsilon}$ is
\begin{align*}
&\ll_\eps (\log q)^\varepsilon \sum_{b \leq X}\frac{\tau(b)^2}{b}\sum_{\ell > q^{\kappa - 2\varepsilon}}\frac{1}{\ell^2}\sum_{an = \ell^2} \tau(a){\tau(n)}\ll_\eps (\log q)^{4+\varepsilon}\sum_{\ell > q^{\kappa - 2\varepsilon}}\frac{\tau(\ell^2)^3}{\ell^2}\ll_\eps q^{-\varepsilon}.
\end{align*}
We therefore have
\begin{align*}
\mathcal{S}(0) &= \sum_{m \leq q^{\varepsilon}}\frac{1}{m^2}\sum_{\substack{ab \leq X}}\frac{\rho_1(ab)(1\star\psi)(bn)}{abn}\sum_{\substack{\ell \geq 1 \\ \ell^2 = an}} 1 + O_\eps(q^{-\varepsilon}) \\
&= \sum_{m \leq q^{\varepsilon}}\frac{1}{m^2}\sum_{\substack{ab \leq X \\ an = \square}}\frac{\rho_1(ab)(1\star\psi)(bn)}{abn}+ O_\eps(q^{-\varepsilon}) \\
&= \zeta(2)\sum_{\substack{ab \leq X \\ an = \square}}\frac{\rho_1(ab)(1\star\psi)(bn)}{abn}+ O_\eps(q^{-\varepsilon}).
\end{align*}

We wish to compare the triple sum over $a,b$, and $n$ with the Euler product
\begin{align*}
\mathcal{P}_0 &:= \prod_{p \leq X}\sum_{\substack{a,b,n \geq 0 \\ a+n \text{ even}}}\frac{\rho_1(p^{a+b})(1\star \psi)(p^{b+n})}{p^{a+b+n}} = \sum_{\substack{p \mid abn \Rightarrow p\leq X \\ an = \square}}\frac{\rho_1(ab)(1\star\psi)(bn)}{abn}.
\end{align*}
It shall follow from lacunarity that $\mathcal{P}_0$ is a good approximation to $\mathcal{S}(0)$.

We use a crude smooth number estimate, as in the proof of Lemma \ref{lemmaB1}, to first restrict the size of $ab$. If we set $T = \exp ((\log X)(\log \log X)^2)$, then one finds, as in the proof of Lemma \ref{lemmaB1}, that the contribution to $\mathcal{P}_0$ from $ab > T$ is $O_A( (\log q)^{-A})$. It follows that
\begin{align*}
\mathcal{P}_0 &= \sum_{\substack{p \mid abn \Rightarrow p\leq X \\ an = \square \\ ab \leq T}}\frac{\rho_1(ab)(1\star\psi)(bn)}{abn} + O_A((\log q)^{-A}).
\end{align*}

The next step is to use lacunarity to show that the error from restricting to $ab \leq X$ is negligible. We recall here that $\rho_1$ is supported on cube-free integers, and that
\begin{align*}
\rho_1(p) &= h(p)\alpha(p)\mu(p)(1\star \psi)(p), \\
\rho_1(p^2) &= h(p)\alpha(p)^2\psi(p).
\end{align*}
The contribution to $\mathcal{P}_0$ from $ab > X$ is
\begin{align*}
&\leq \sum_{\substack{p \mid abn \Rightarrow p\leq X \\ X <ab \leq T}}\frac{|\rho_1(ab)|\tau(b)\tau(n)}{abn} \ll (\log q)^2\sum_{\substack{p \mid t \Rightarrow p \leq X \\ X < t \leq T}}\frac{|\rho_1(t)|\tau(t)^2}{t}.
\end{align*}
Since $\rho_1$ is supported on cube-free integers we can factor $t = cd^2$, where $c,d$ are squarefree and $(c,d)=1$. Therefore this last sum is
\begin{align*}
&\ll (\log q)^2 \sum_{X < cd^2 \leq T} \frac{h(c)h(d)\alpha(c)\alpha(d)^2(1\star\psi)(c)\tau(c)^2\tau(d^2)^2}{cd^2} \\
&\ll_\eps (\log q)^{2+\varepsilon}\sum_{X < cd^2 \leq T} \frac{(1\star\psi)(c)\tau(c)^2\tau(d^2)^2}{cd^2}.
\end{align*}
Since $cd^2 > X$ we either have $c > X^{1/2}$ or $d > X^{1/4}$. The contribution from $d > X^{1/4}$ is
\begin{align*}
&\ll_\eps(\log q)^{2+\varepsilon} \sum_{d > X^{1/4}}\frac{\tau(d^2)^2}{d^2}\sum_{c \leq T} \frac{\tau(c)^3}{c}\ll_\eps(\log q)^{10+\varepsilon} \sum_{d > X^{1/4}}\frac{\tau(d^2)^2}{d^2}\ll_\eps X^{-1/4+\varepsilon}.
\end{align*}
To estimate the contribution from $c > X^{1/2}$ we use Lemma \ref{lem:updated lacunarity lemma}. It follows that the contribution from $c > X^{1/2}$ is
\begin{align*}
&\ll_\eps(\log q)^{2+\varepsilon}\sum_{X^{1/2} < c \leq T}\frac{(1\star\psi)(c)\tau(c)^2}{c} \ll_{\varepsilon,A}L(1,\psi)(\log q)^{11+\varepsilon} + (\log q)^{-A}.
\end{align*}
We have therefore shown that
\begin{align*}
\mathcal{P}_0 &= \sum_{\substack{ab\leq X \\ p\mid n \Rightarrow p \leq X\\ an = \square}}\frac{\rho_1(ab)(1\star\psi)(bn)}{abn} + O_\varepsilon(L(1,\psi)(\log q)^{11+\varepsilon}) +O_A((\log q)^{-A}).
\end{align*}

The sum here is almost equal to the triple sum in $\mathcal{S}(0)$, save for the condition on primes dividing $n$. However, if there exists a prime $p > X$ dividing $n$ then we must have $p^2 \mid n$, since $an$ is a square and $a\leq X$. An easy argument then shows that
\begin{align*}
\mathcal{P}_0 &=\sum_{\substack{ab \leq X \\ an = \square}}\frac{\rho_1(ab)(1\star\psi)(bn)}{abn}+ O_\varepsilon(L(1,\psi)(\log q)^{11+\varepsilon}) +O_A((\log q)^{-A}),
\end{align*}
and therefore
\begin{align*}
\mathcal{S}(0) &= \zeta(2)\prod_{p \leq X}\sum_{\substack{a,b,n \geq 0 \\ a+n \text{ even}}}\frac{\rho_1(p^{a+b})(1\star \psi)(p^{b+n})}{p^{a+b+n}}+ O_\varepsilon(L(1,\psi)(\log q)^{11+\varepsilon}) +O_A((\log q)^{-A}).
\end{align*}

We now follow the familiar procedure of considering the local factors depending on the value of $\psi(p)$. If $\psi(p) = 0$, i.e. if $p \mid D$, then
\begin{align*}
\sum_{\substack{a,b,n \geq 0 \\ a+n \text{ even}}}\frac{\rho_1(p^{a+b})(1\star \psi)(p^{b+n})}{p^{a+b+n}}& = \sum_{\substack{a,b,n \geq 0 \\ a+n \text{ even}}}\frac{\alpha(p^{a+b})\mu(p^{a+b})}{p^{a+b+n}} \\
&=\sum_{n \text{ even}}\frac{1}{p^n} + \frac{\alpha(p)\mu(p)}{p}\sum_{n \geq 1 \text{\,odd}}\frac{1}{p^n} + \frac{\alpha(p)\mu(p)}{p}\sum_{n \text{ even}}\frac{1}{p^n} \\
&= \frac{p^2}{p^2-1}-\frac{\alpha(p)}{p^2-1}-\alpha(p) \frac{p}{p^2-1}= \left(1 + \frac{1}{p}\right)^{-1}.
\end{align*}

We next examine the local factor when $\psi(p) = -1$. In this case we have
\begin{align*}
&\sum_{\substack{a,b,n \geq 0 \\ a+n \text{ even}}}\frac{\rho_1(p^{a+b})(1\star \psi)(p^{b+n})}{p^{a+b+n}} =\sum_{\substack{a,b,n \geq 0 \\ a+n \text{ even} \\ b+n \text{ even} \\ a+b \in \{0,2\}}}\frac{h(p^{a+b})\alpha^2(p^{a+b})\psi(p^{(a+b)/2})}{p^{a+b+n}}\\
&\qquad\qquad=\sum_{n \text{ even}}\frac{1}{p^n} -\frac{h(p)\alpha(p)^2}{p^2}\sum_{n \geq 1 \text{\,odd}}\frac{1}{p^n}-\frac{2h(p)\alpha(p)^2}{p^2}\sum_{n \text{ even}}\frac{1}{p^n}\\
&\qquad\qquad= \frac{p^2}{p^2-1} -\frac{h(p)\alpha(p)^2}{p(p^2-1)}-\frac{2h(p)\alpha(p)^2}{p^2-1} = \frac{p^2+p}{p^2+p+1}\\
&\qquad\qquad= \left(\left(1 + \frac{1}{p} \right)^2 + \frac{\psi(p)}{p} \right)^{-1}\left(1 - \frac{\psi(p)}{p} \right).
\end{align*}

Lastly, we examine the local factor when $\psi(p) = 1$:
\begin{align*}
&\sum_{\substack{a,b,n \geq 0 \\ a+n \text{ even}}}\frac{\rho_1(p^{a+b})(1\star \psi)(p^{b+n})}{p^{a+b+n}} = \sum_{\substack{a,b,n \geq 0 \\ a+n \text{ even}}}\frac{\rho_1(p^{a+b})(b+n+1)}{p^{a+b+n}}\\
&\qquad=\sum_{n \text{ even}}\frac{n+1}{p^n}-\frac{2h(p)\alpha(p)}{p}\sum_{n \geq 1 \text{\,odd}}\frac{n+1}{p^n}-\frac{2h(p)\alpha(p)}{p}\sum_{n\text{ even}}\frac{n+2}{p^n}\\
&\qquad\qquad+\frac{h(p)\alpha(p)^2}{p^2}\sum_{n \geq 1 \text{\,odd}}\frac{n+2}{p^n}+\frac{h(p)\alpha(p)^2}{p^2}\sum_{n \text{ even}}\frac{n+1}{p^n}+\frac{h(p)\alpha(p)^2}{p^2}\sum_{n \text{ even}}\frac{n+3}{p^n}\\
&\qquad=\frac{p^2(p^2+1)}{(p^2-1)^2}- \frac{4h(p)\alpha(p)p^2}{(p^2-1)^2}-\frac{4h(p)\alpha(p)p^3}{(p^2-1)^2} +\frac{h(p)\alpha(p)^2(3p^2-1)}{p(p^2-1)^2}\\
&\qquad\qquad+\frac{h(p)\alpha(p)^2(p^2+1)}{(p^2-1)^2}+\frac{h(p)\alpha(p)^2(3p^2-1)}{(p^2-1)^2}\\
&\qquad=\frac{p^2-p}{p^2+3p+1} = \left(\left(1 + \frac{1}{p} \right)^2 + \frac{\psi(p)}{p} \right)^{-1}\left(1 - \frac{\psi(p)}{p} \right).
\end{align*}

We see that the local factors match in each case with those in $\mathfrak{S}_1$, and this completes the proof.
\end{proof}

The argument for the second moment is similar, but the convexity bound is no longer sufficient to obtain \eqref{crit2}. Instead, we use the subconvexity bound
\begin{align*}
L\Big(f,\frac{1}{2} \Big)\ll_\varepsilon q^{5/24 + \vartheta/12+\varepsilon}
\end{align*}
due to Blomer and Khan [\textbf{\ref{BK}}; Theorem 4]. Here $\vartheta$ is an admissible bound towards the Ramanujan conjecture, and it is known that $\vartheta \leq 7/64$ is admissible (see [\textbf{\ref{BB}}] for more discussion). This yields the subconvexity bound
\begin{align*}
L\Big(f,\frac{1}{2} \Big)\ll_\varepsilon q^{167/768+\varepsilon}.
\end{align*}
For the central value $L(f \otimes \psi, 1/2)$ we use the convexity bound
\begin{align*}
L\Big(f \otimes \psi, \frac{1}{2}\Big) \ll_\varepsilon (qD^2)^{1/4+\varepsilon}.
\end{align*}
Recalling that we choose $X$ in the definition of the mollifier so that $X=D^{24}$, we see that \eqref{crit2} is satisfied if $q > D^{380}$.

\end{document}